\documentclass[10pt,reqno]{amsart}
\usepackage{amsmath}
\usepackage{amsthm}
\usepackage{amssymb}
\usepackage{graphicx}
\usepackage{amsfonts}
\usepackage{latexsym}
\usepackage{hyperref,subcaption}

\graphicspath{{IMAGES/}}

\usepackage{enumitem}
\setlist[itemize]{leftmargin=*}

\usepackage[font=footnotesize,labelfont=bf]{caption}
\usepackage[font=scriptsize,labelfont=bf]{subcaption}

\usepackage[table]{xcolor}
\usepackage{cite}
\usepackage{multicol}

\theoremstyle{plain}
\numberwithin{equation}{subsection}
\newtheorem{theorem}[equation]{Theorem}
\newtheorem{corollary}[equation]{Corollary}
\newtheorem{lemma}[equation]{Lemma}

\theoremstyle{definition}
\newtheorem{example}[equation]{Example}

\newcommand{\highlight}[1]{{\color{blue} #1}}

\newcommand{\Q}{\mathbb{Q}}

\newcommand{\N}{\mathbb{N}}
\newcommand{\K}{\mathbb{K}}
\newcommand{\Z}{\mathbb{Z}}
\newcommand{\C}{\mathbb{C}}

\newcommand{\li}{\operatorname{li}}
\newcommand{\Li}{\operatorname{Li}}

\renewcommand{\pmod}[1]{\,\,(\operatorname{mod}#1)}

\let\oldenumerate=\enumerate
	\def\enumerate{
	\oldenumerate
	\setlength{\itemsep}{5pt}
	}
\let\olditemize=\itemize
	\def\itemize{
	\olditemize
	\setlength{\itemsep}{5pt}
	}

\allowdisplaybreaks

\begin{document}

\title[The Bateman--Horn Conjecture]{The Bateman--Horn Conjecture:  Heuristics, History, and Applications}

	\author[S.L.~Aletheia-Zomlefer]{Soren Laing Aletheia-Zomlefer}

\author[L.~Fukshansky]{Lenny Fukshansky}
\address{Department of Mathematics,
Claremont McKenna College,
850 Columbia Ave,
Claremont, CA 91711,
USA}
\email{lenny@cmc.edu}

	\author[S.R.~Garcia]{Stephan Ramon Garcia}
	\address{Department of Mathematics, Pomona College, 610 N. College Ave., Claremont, CA 91711} 
	\email{stephan.garcia@pomona.edu}
	\urladdr{\url{http://pages.pomona.edu/~sg064747}}
	
\thanks{SRG supported by a David L.~Hirsch III and Susan H.~Hirsch Research Initiation Grant, the Institute for Pure and Applied Mathematics (IPAM) Quantitative Linear Algebra program, and NSF Grant DMS-1800123. LF supported by the Simons Foundation grant \#519058.}



\begin{abstract}
The Bateman--Horn conjecture is a far-reaching statement about the distribution of the prime numbers.  It implies many known results, such as the prime number theorem and the Green--Tao theorem, along with many famous conjectures, such the twin prime conjecture and Landau's conjecture.  We discuss the Bateman--Horn conjecture, its applications, and its origins.  
\end{abstract}

\maketitle

\def\A{{\mathcal A}}
\def\AA{{\mathfrak A}}
\def\B{{\mathcal B}}
\def\C{{\mathcal C}}
\def\D{{\mathcal D}}
\def\EE{{\mathfrak E}}
\def\F{{\mathcal F}}
\def\G{{\mathcal G}}
\def\x{{\mathcal H}}
\def\I{{\mathcal I}}
\def\II{{\mathfrak I}}
\def\J{{\mathcal J}}
\def\kk{{\mathfrak K}}
\def\L{{\mathcal L}}
\def\LL{{\mathfrak L}}
\def\M{{\mathcal M}}
\def\mm{{\mathfrak m}}
\def\MM{{\mathfrak M}}
\def\O{{\mathcal O}}
\def\OO{{\mathfrak O}}
\def\P{{\mathbb P}}
\def\PP{{\mathfrak p}}
\def\W{{\mathcal W}}
\def\PNR{{\mathcal P_N(\real)}}
\def\PMNR{{\mathcal P^M_N(\real)}}
\def\PdNR{{\mathcal P^d_N(\real)}}
\def\s{{\mathcal S}}
\def\V{{\mathcal V}}
\def\X{{\mathcal X}}
\def\Y{{\mathcal Y}}
\def\H{{\mathcal H}}
\def\cee{{\mathbb C}}
\def\Nn{{\mathbb N}}
\def\pee{{\mathbb P}}
\def\Q{{\mathbb Q}}
\def\QQ{{\mathbb Q}}
\def\real{{\mathbb R}}
\def\RR{{\mathbb R}}
\def\Z{{\mathbb Z}}
\def\ZZ{{\mathbb Z}}
\def\aaa{{\mathbb A}}
\def\ff{{\mathbb F}}
\def\HDelta{\emph{\Delta}}
\def\kk{{\mathfrak K}}
\def\qbar{{\overline{\mathbb Q}}}
\def\kbar{{\overline{K}}}
\def\ybar{{\overline{Y}}}
\def\kkbar{{\overline{\mathfrak K}}}
\def\ubar{{\overline{U}}}
\def\eps{{\varepsilon}}
\def\ahat{{\hat \alpha}}
\def\bhat{{\hat \beta}}
\def\k{{\nu}}
\def\gt{{\tilde \gamma}}
\def\h{{\tfrac12}}
\def\be{{\boldsymbol e}}
\def\bei{{\boldsymbol e_i}}
\def\bc{{\boldsymbol c}}
\def\bdt{{\boldsymbol \delta}}
\def\bff{{\boldsymbol f}}
\def\bm{{\boldsymbol m}}
\def\bk{{\boldsymbol k}}
\def\bi{{\boldsymbol i}}
\def\bl{{\boldsymbol l}}
\def\bq{{\boldsymbol q}}
\def\bu{{\boldsymbol u}}
\def\bt{{\boldsymbol t}}
\def\bs{{\boldsymbol s}}
\def\bv{{\boldsymbol v}}
\def\bw{{\boldsymbol w}}
\def\bx{{\boldsymbol x}}
\def\bbx{{\overline{\boldsymbol x}}}
\def\bX{{\boldsymbol X}}
\def\bz{{\boldsymbol z}}
\def\bwy{{\boldsymbol y}}
\def\bY{{\boldsymbol Y}}
\def\bL{{\boldsymbol L}}
\def\ba{{\boldsymbol a}}
\def\bb{{\boldsymbol b}}
\def\bet{{\boldsymbol\eta}}
\def\bxi{{\boldsymbol\xi}}
\def\bo{{\boldsymbol 0}}
\def\bone{{\boldsymbol 1}}
\def\bol{{\boldsymbol 1}_L}
\def\ep{\varepsilon}
\def\p{\boldsymbol\varphi}
\def\q{\boldsymbol\psi}
\def\rank{\operatorname{rank}}
\def\aut{\operatorname{Aut}}
\def\lcm{\operatorname{lcm}}
\def\sgn{\operatorname{sgn}}
\def\spn{\operatorname{span}}
\def\md{\operatorname{mod}}
\def\Norm{\operatorname{Norm}}
\def\dim{\operatorname{dim}}
\def\det{\operatorname{det}}
\def\Vol{\operatorname{Vol}}
\def\rk{\operatorname{rk}}
\def\ord{\operatorname{ord}}
\def\ker{\operatorname{ker}}
\def\div{\operatorname{div}}
\def\Gal{\operatorname{Gal}}
\def\GL{\operatorname{GL}}
\def\SNR{\operatorname{SNR}}
\def\WR{\operatorname{WR}}
\def\IWR{\operatorname{IWR}}
\def\scg{\operatorname{\langle  \Gamma \rangle }}
\def\swrh{\operatorname{Sim_{WR}(\Lambda_h)}}
\def\ch{\operatorname{C_h}}
\def\cht{\operatorname{C_h(\theta)}}
\def\scgt{\operatorname{\langle  \Gamma_{\theta} \rangle }}
\def\scgmn{\operatorname{\langle  \Gamma_{m,n} \rangle }}
\def\gat{\operatorname{\Omega_{\theta}}}
\def\Obar{\operatorname{\overline{\Omega}}}
\def\Lbar{\operatorname{\overline{\Lambda}}}
\def\mn{\operatorname{mn}}
\def\disc{\operatorname{disc}}
\def\rot{\operatorname{rot}}
\def\Prob{\operatorname{Prob}}
\def\co{\operatorname{co}}
\def\ot{\operatorname{o_{\tau}}}
\def\Aut{\operatorname{Aut}}
\def\Mat{\operatorname{Mat}}
\def\SL{\operatorname{SL}}
\def\id{\operatorname{id}}

\tableofcontents

\section{Introduction}
\label{intro}

Given a collection of polynomials with integer coefficients, how often should we expect their values at integer arguments to be simultaneously prime? This general question 
subsumes a large number of different directions and investigations in analytic number theory. 
A comprehensive answer is proposed by the famous Bateman--Horn conjecture, first formulated by Paul T.~Bateman and Roger A.~Horn in 1962 \cite{Bateman, BatemanHorn-2}. This conjecture is a far-reaching statement about the distribution of the prime numbers.  Many well-known theorems, such as the prime number theorem and the Green--Tao theorem, follow from it. The conjecture also implies a variety of unproven conjectures, such as the twin prime conjecture and Landau's conjecture.  We hope to convince the reader
that the Bateman--Horn conjecture deserves to be ranked among the Riemann hypothesis and $abc$-conjecture as one of the most important unproven conjectures in number theory.

The amount of literature related to the Bateman--Horn conjecture is large: MathSciNet, for example, shows over 100 citations to the original Bateman--Horn papers in which the conjecture was formulated. Somewhat surprisingly, however, we did not find many expository accounts besides a short note by Serge Lang~\cite{lang_bh} with just a quick overview of the conjecture.  It is a goal of this paper to provide a detailed exposition of the conjecture and some of its ramifications. 
We assume no knowledge beyond elementary undergraduate number theory.  We introduce
the necessary algebraic and analytic prerequisites as need arises.
We do not attempt a comprehensive survey of all the literature related to the Bateman--Horn conjecture. 
For example, recent variations of the conjecture, say to multivariate polynomials \cite{moroz,Destagnol} or to polynomial rings over finite fields \cite{Conrad, Conrad2},
are not treated here.

The organization of this paper is as follows. Section~\ref{Section:Preliminaries} introduces asymptotic equivalence, the logarithmic integral, and the prime number theorem.  In Section~\ref{Section:Heuristic}, we go through a careful heuristic argument based upon the Cram\'er model that explains 
most of the key restrictions and predictions of the Bateman--Horn conjecture.  
Before proceeding to various examples and applications of the conjecture, Section~\ref{Section:History} revisits some of the historical background.  In particular, we include many personal recollections
of Roger Horn that have never before been published.

One of the main features of the Bateman--Horn conjecture is an explicit constant in the main term of the asymptotic formula for the number of integers below a given threshold at which a collection of polynomials simultaneously assume prime values. The expression for this constant, however, is complicated and involves an infinite product. It is nontrivial to see that this product converges and we are not aware of a detailed proof of this fact anywhere in the literature.  The original Bateman--Horn paper sketches the main idea of this proof, but omits almost all of the details. We present this argument in detail in Section~\ref{Section:Converge}.

Section~\ref{Section:Single} is devoted to a number of important instances and consequences of the single polynomial case of the conjecture, while ramifications of the multiple polynomial case are discussed in Section \ref{Section:Multiple}. Finally, we discuss some limitations of the Bateman--Horn conjecture in Section \ref{Section:Limitations}. With this brief introduction, we are now ready to proceed.

\medskip\noindent\textbf{Acknowledgments}.  
We thank 
Keith Conrad for many technical corrections,
Harold G.~Diamond for permitting us to use two photographs of Paul Bateman,
Jeff Lagarias for several suggestions about the exposition,
Florian Luca for introducing us to the Bateman--Horn conjecture,
and Hugh Montgomery for his remarks about Bateman.
We especially thank Roger A.~Horn for supplying us with his extensive recollections and several photographs,
and for many comments on an initial draft of this paper.
Special thanks goes to the anonymous referee for suggesting dozens of improvements to the exposition.

\medskip\noindent\textbf{Disclaimer}.
This paper originally appeared on the \texttt{arXiv} under the title ``One conjecture to rule them all:  Bateman--Horn'' (\url{https://arxiv.org/abs/1807.08899}).  

\section{Preliminaries}\label{Section:Preliminaries}
We will often need to compare the rate of growth of two real-valued functions of a real variable as their arguments tend to infinity.  To this end, we require a bit of notation.  Readers familiar with asymptotic equivalence, Big-$O$ and little-$o$ notation, and the prime number theorem should proceed to Section \ref{Section:Heuristic}. A good source of information on classical analytic number theory is \cite{luca}.

\subsection{Asymptotic equivalence}
In what follows, we assume that
$f(x)$ and $g(x)$ are continuous, 
real-valued functions that are defined and nonzero for sufficiently large $x$.
We write $f \sim g$ to mean that
\begin{equation}\label{eq:Asymptotic}
\lim_{x\to\infty} \frac{f(x)}{g(x)} = 1.
\end{equation}
We say that $f$ and $g$ are \emph{asymptotically equivalent} when this occurs.
The limit laws from calculus show that $\sim$ is an equivalence relation; we use this
fact freely.

Two polynomials are asymptotically equivalent if and only if they have the same degree and 
the same leading coefficient.  For example, $2x^2 \sim 2x^2 + x + 1$ since 
\begin{equation*}
\lim_{x\to\infty} \frac{2x^2+x+1}{2x^2} = \lim_{x\to\infty} \left(1 + \frac{1}{x} + \frac{1}{x^2} \right) = 1.
\end{equation*}
It is important to note, however, that asymptotic equivalence does not necessarily mean
that ``$f$ and $g$ are close together'' in the sense that $f-g$ is small.
Although $2x^2 \sim 2x^2+x+1$, their difference
$(2x^2+x+1)-2x^2 = x+1$ is unbounded.

\subsection{Big-$O$ and little-$o$ notation}\label{Section:O}
When we write $f(x) = O(g(x))$, we mean that there is a constant $C$ such that
$|f(x)|\leq C|g(x)|$ for sufficiently large $x$.  For example,
\begin{equation*}
4x^2+7x \log x = O(x^2)
\qquad \text{and}\qquad
\sin x = O(1).
\end{equation*}
What is the relationship between Big-$O$ notation and asymptotic equivalence?
If $f\sim g$, then $f(x) = O(g(x))$ and $g(x) = O(f(x))$.  Indeed, \eqref{eq:Asymptotic} and the definition of limits ensures that 
$|f(x)| \leq 2 |g(x)|$ and $|g(x)| \leq 2|f(x)|$ for sufficiently large $x$ (the number $2$ in the preceding inequalities can be replaced by any constant greater than $1$).
On the other hand, $2x = O(x)$ and $x = O(2x)$, although $x$ and $2x$ are not asymptotically equivalent. Hence the statement ``$f \sim g$" is stronger than the statement ``$f(x) = O(g(x))$ and $g(x) = O(f(x))$", but both of these statements are asymptotic in their nature.

We say $f(x) = o(g(x))$ if
\begin{equation*}
\lim_{x \to \infty} \frac{f(x)}{g(x)} = 0.
\end{equation*}
For instance, $x = o(x^2)$ as $x \to \infty$. Notice that if $f \sim g$, then
\begin{equation*}
1 = \lim_{x \to \infty} \frac{f(x)}{g(x)} = \lim_{x \to \infty} \frac{f(x) - g(x) + g(x)}{g(x)} =  \lim_{x \to \infty} \frac{f(x) - g(x)}{g(x)} + 1,
\end{equation*}
and so $\lim_{x \to \infty} \frac{f(x) - g(x)}{g(x)} = 0$.  Thus,
the error term satisfies $|f(x) - g(x)| = o(g(x))$. On the other hand, the assertion that $f(x) = O(g(x))$ and $g(x) = O(f(x))$ does not guarantee a smaller order error term.
Indeed, $x = O(2x)$ and $2x = O(x)$, but $|x-2x| = |x|$ is not $o(x)$ or $o(2x)$.

\subsection{The logarithmic integral}
In the theory of prime numbers
the \emph{offset logarithmic integral}\footnote{The function \eqref{eq:Li} is a close relative of the standard 
\emph{logarithmic integral} $\operatorname{li}(x)$,
in which the lower limit of integration in \eqref{eq:Li} is $0$ and the singularity 
at $x=1$ is avoided by using a Cauchy principal value.  Since we are interested in large $x$, we use \eqref{eq:Li} instead.} 
\begin{equation}\label{eq:Li}
\Li(x) = \int_2^x \frac{dt}{\log t}
\end{equation}
and its close relatives frequently arise.  Here $\log t$ denotes the natural logarithm\footnote{The
notation $\ln t$ may be more familiar to calculus students.} of $t$.
Unfortunately, the integral \eqref{eq:Li} cannot be evaluated in closed form.
As a consequence, it is convenient to replace $\Li(x)$ and its relatives (see Figure \ref{Figure:LogIntegrals})
with simpler functions that are asymptotically equivalent.

\begin{figure}
      \centering
      \includegraphics[width=0.75\textwidth]{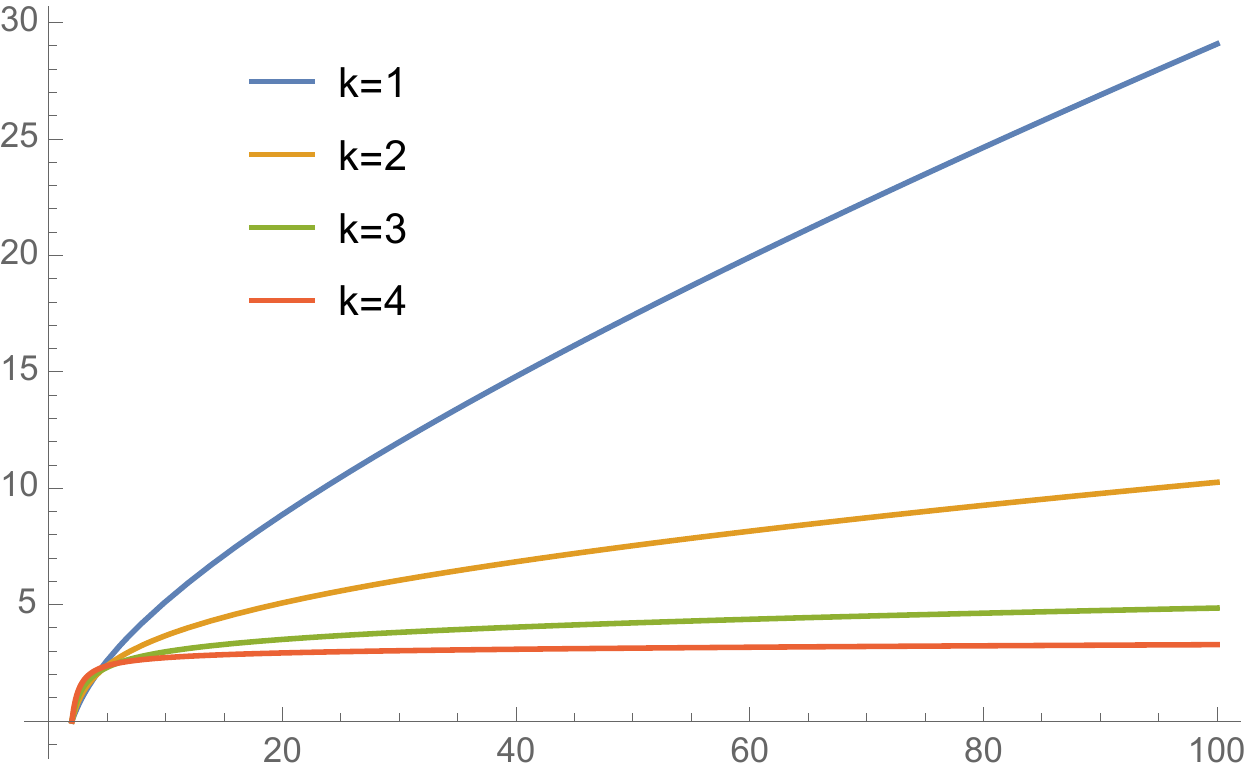}
      \caption{The functions $\displaystyle\int_2^x \frac{dt}{(\log t)^k}$ for $k=1,2,3,4$ for $x \geq 2$.}
      \label{Figure:LogIntegrals}
\end{figure}

\begin{lemma}\label{Lemma:Log}
$\displaystyle \int_2^x \frac{dt}{(\log t)^k} \sim \frac{x}{(\log x)^k}$
for $k=1,2,\ldots$.
\end{lemma}

\begin{proof}
L'H\^opital's rule and the fundamental theorem of calculus imply that
\begin{equation*}
\lim_{x\to\infty} \frac{  \int_2^x \frac{dt}{(\log t)^k} }{x/(\log x)^k}
\overset{L}{=} \lim_{x\to\infty} \frac{1/(\log x)^k}{1/(\log x)^k - k/ (\log x)^{k+1}} 
= \lim_{x\to\infty}  \frac{1}{1 - k/\log x} 
=1.\qedhere
\end{equation*}
\end{proof}

One can be a little more precise than Lemma \ref{Lemma:Log}.
Integration by parts provides:
\begin{equation*}
\Li(x) = \frac{x}{\log x} + O\left( \frac{x}{(\log x)^2} \right)
\end{equation*}
and
\begin{equation*}
\int_2^x \frac{dt}{(\log t)^k} = \frac{x}{(\log x)^k} + O\left( \frac{x}{(\log x)^{k+1}} \right).
\end{equation*}

\subsection{Prime number theorem}\label{Section:PNT}
The first signpost toward the Bateman--Horn conjecture is the prime number theorem,
which describes the gross distribution of the primes.
Let $\pi(x)$ denote the number of primes at most $x$.
For example, $\pi(10.5) = 4$ since $2,3,5,7 \leq 10.5$.
The following result was  proved independently by Hadamard and de la Vall\'ee Poussin in 1896;
see Figure \ref{Figure:PNT_Li}.

\begin{figure}
    \centering
    \begin{subfigure}[t]{0.475\textwidth}
        \centering
        \includegraphics[width=\textwidth]{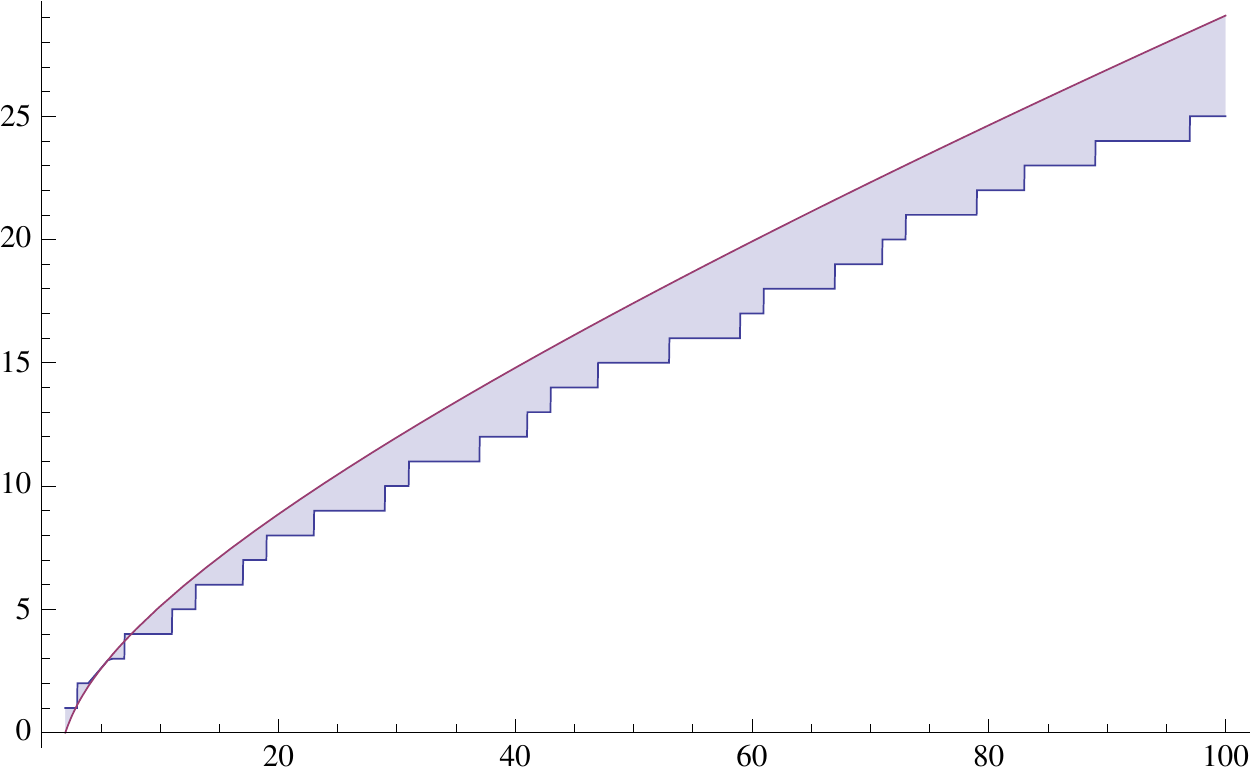}
        \caption{$x \leq 100$}
    \end{subfigure}
    \begin{subfigure}[t]{0.475\textwidth}
        \centering
        \includegraphics[width=\textwidth]{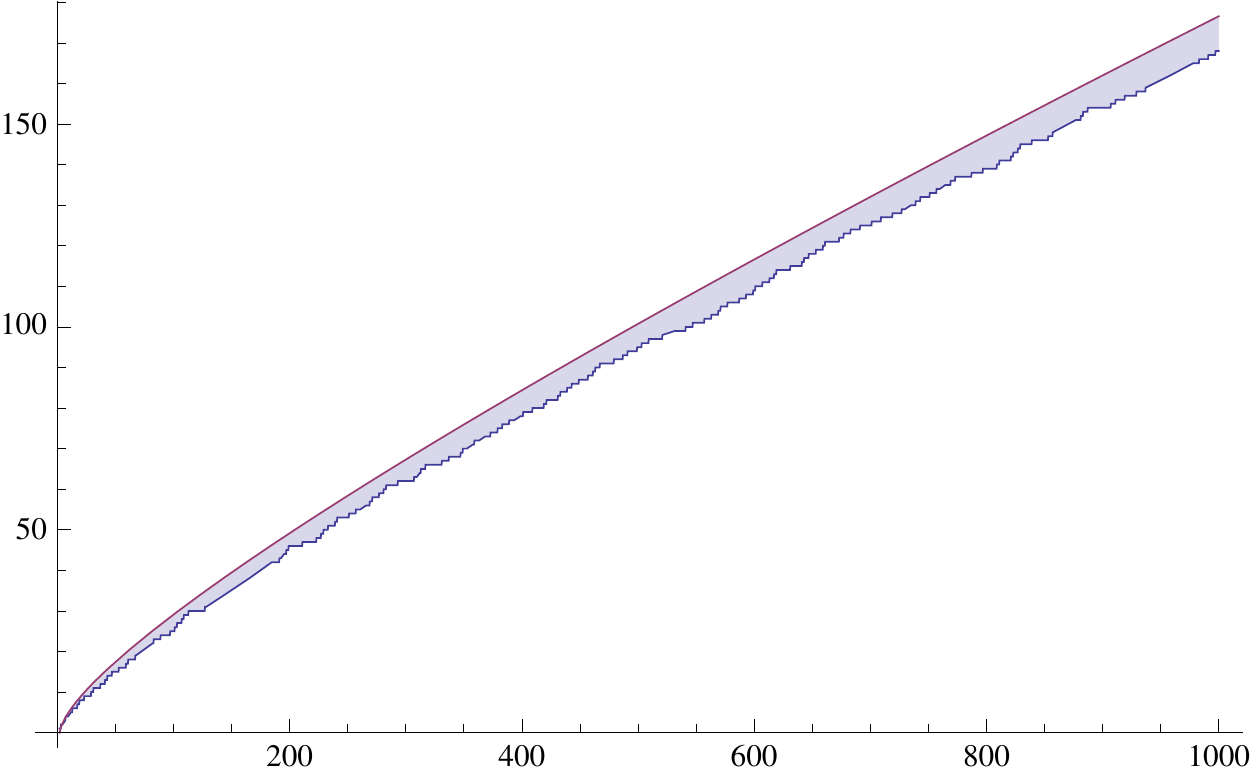}
        \caption{$x \leq 1{,}000$}
    \end{subfigure}
\\
    \begin{subfigure}[t]{0.475\textwidth}
        \centering
        \includegraphics[width=\textwidth]{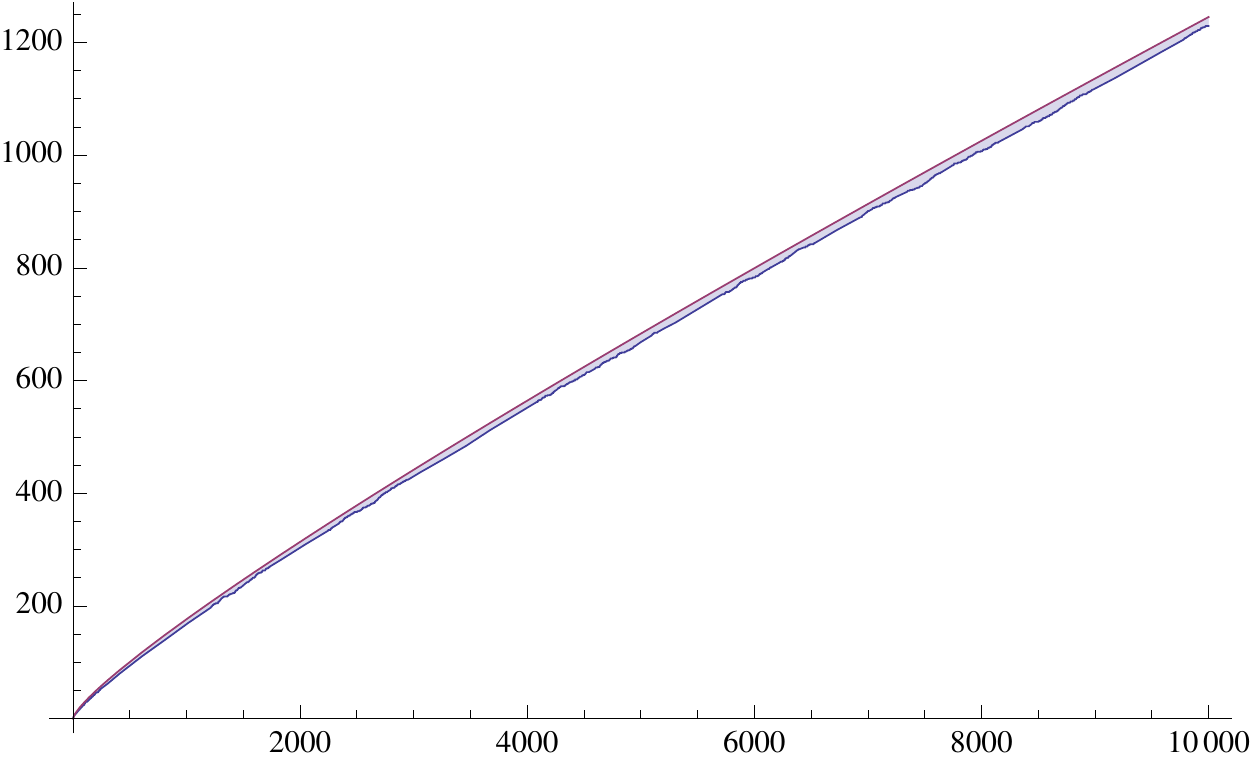}
        \caption{$x \leq 10{,}000$}
    \end{subfigure}
    \begin{subfigure}[t]{0.475\textwidth}
        \centering
        \includegraphics[width=\textwidth]{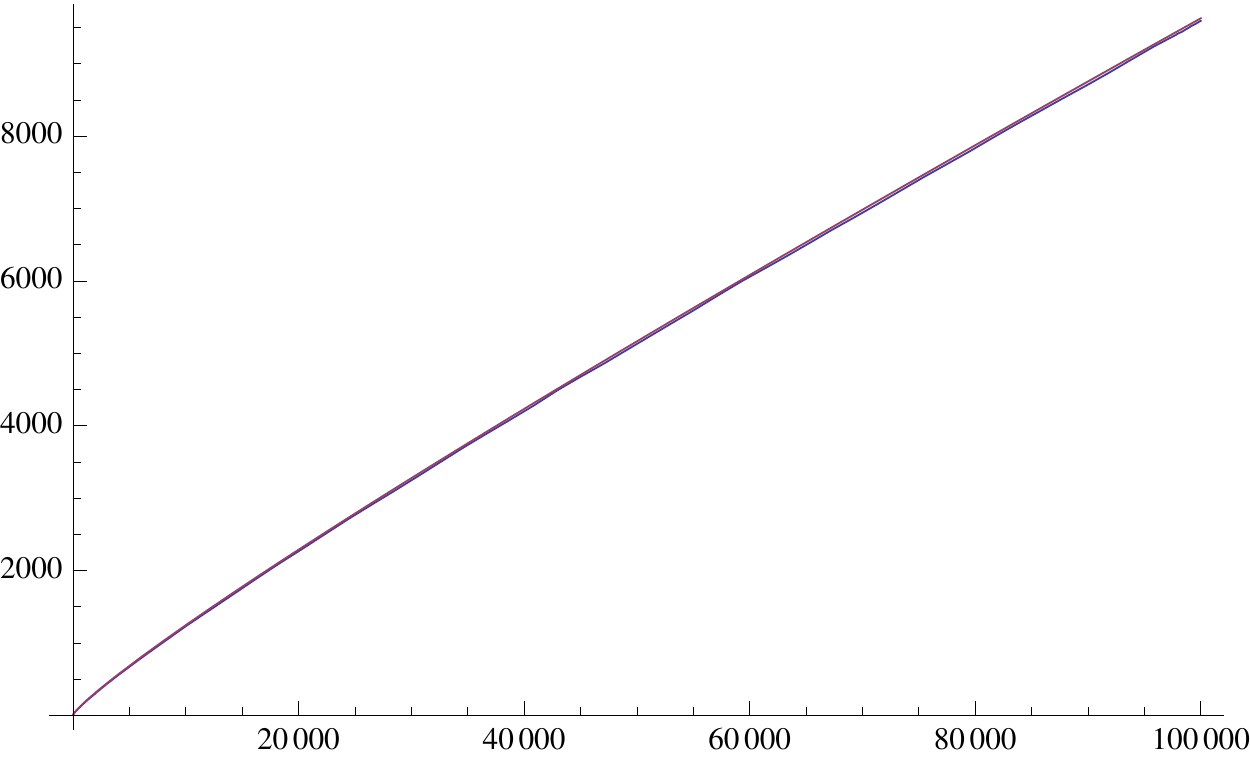}
        \caption{$x \leq 100{,}000$}
    \end{subfigure}
    \caption{Graphs of $\operatorname{Li}(x)$ versus $\pi(x)$ on various scales. }
    \label{Figure:PNT_Li}
\end{figure}

\begin{theorem}[Prime Number Theorem]\label{Theorem:PNT}
$\displaystyle \pi(x) \sim \Li(x)$.
\end{theorem}

Although $\Li(x) \sim x/\log x$,
the logarithmic integral provides a more accurate approximation to $\pi(x)$;
see Table \ref{Table:LogInt}.  For simplicity, we work now with the approximation $\pi(x) \sim x/\log x$ and develop
a probabilistic model of the prime numbers that will guide our progress toward the
Bateman--Horn conjecture \cite{TaoCramer}.

\begin{table}
    \begin{tabular}{|c||c|c|c|}
    \hline
    $x$ & $\pi(x)$ & $\operatorname{Li}(x)$ & $x / \log x$ \\
    \hline
    1000 & 168 & 177 & 145 \\
    10{,}000 & 1{,}229 &  1{,}245 & 1{,}086 \\
    100{,}000 &  9{,}592 & 9{,}629 & 8{,}686 \\
    1{,}000{,}000 & 78{,}498 & 78{,}627 & 72{,}382\\
    10{,}000{,}000 & 664{,}579 & 664{,}917 & 620{,}421 \\
    100{,}000{,}000 & 5{,}761{,}455 & 5{,}762{,}208 & 5{,}428{,}681 \\
    1{,}000{,}000{,}000 & 50{,}847{,}534 & 50{,}849{,}234 & 48{,}254{,}942\\
    10{,}000{,}000{,}000 & 455{,}052{,}511 & 455{,}055{,}614 & 434{,}294{,}482 \\
    100{,}000{,}000{,}000 & 4{,}118{,}054{,}813 & 4{,}118{,}066{,}400 &  3{,}948{,}131{,}654 \\
    1{,}000{,}000{,}000{,}000 & 37{,}607{,}912{,}018 & 37{,}607{,}950{,}280 & 36{,}191{,}206{,}825 \\
    \hline
    \end{tabular}
    \smallskip
    \caption{The logarithmic integral $\operatorname{Li}(x)$ is a better approximation to the prime counting function
    $\pi(x)$ than is $x / \log x$.  The entries in the table have been rounded to the nearest integer.}
    \label{Table:LogInt}
\end{table}

For fixed $c > 0$ and large $x$, the prime number theorem tells us to expect about
\begin{equation*}
\frac{x+c x }{\log (x+c x)} - 
\frac{x-c x }{\log (x-c x )} 
\ \sim \ \frac{2c x }{\log x}
\end{equation*}
primes in the interval $[x-c x ,x+c x ]$.  
Dividing by the length $2c x$ of the interval, it follows that
the probability that a natural number in the vicinity of $x$ is prime is roughly $1/\log x$.
We use this repeatedly as a guide in our heuristic arguments.

\section{A heuristic argument}\label{Section:Heuristic}

Now that we know about the gross distribution of the primes, it is natural to ask about
the distribution of primes of certain forms.  For example, are there infinitely many primes of the form $n^2+1$?
This was asked at the 1912 International Congress of Mathematicians by Edmund Landau (1877--1938) 
and remains open today.\footnote{Although commonly known as Landau's conjecture, its first appearance is in a 1752 letter from
Leonhard Euler (1707--1783) to Christian Goldbach (1690--1764) \cite[p.2-3]{Euler}.}  

\subsection{A single polynomial}
We let $\Z[x]$ denote the set of polynomials in $x$ with coefficients in $\Z$, the set of integers.  We denote by $\N$ the set
$\{1,2,\ldots\}$ of natural numbers.
For $f \in \Z[x]$, we define
\begin{equation*}
Q(f;x)\, =\, \#\{n \leq x: \text{$f(n)$ is prime}\},
\end{equation*}
in which $\#S$ denotes the number of elements of a set $S$.
We investigate some conditions that $f$ must satisfy if it is to generate infinitely many distinct primes.
To avoid trivialities, suppose that $f$ is nonconstant and that $Q(f;x)\to\infty$.

\begin{itemize}
\item \textbf{Leading coefficient.}
The degree of $f$, denoted $\deg f$, must be at least one.
Moreover, the leading coefficient of $f$ must be positive.

\item \textbf{Irreducible.}
We claim that $f$ is irreducible; that is, it cannot be factored as a product
of two polynomials in $\Z[x]$, neither of which is $\pm 1$.\footnote{Gauss' lemma
ensures that a primitive $f \in \Z[x]$ is irreducible in $\Z[x]$ if and only if it is irreducible in $\Q[x]$, the ring of polynomials with rational coefficients \cite[Prop.~5, p.~303]{DF}.}  
Suppose that $f = gh$, in which $g,h \in \Z[x]$.  Without loss of generality, we may assume that the leading coefficients of $g$ and $h$ are positive.
Then $g(n) = 1$ or $h(n) = 1$ for infinitely many $n$ since $f$ assumes prime values infinitely often.
Consequently, $g-1$ or $h-1$ is a polynomial with infinitely many roots and hence $g$ or $h$ is identically $1$.
Thus, $f$ is irreducible.

\item \textbf{Nonvanishing modulo every prime.}
A nonconstant $f \in \Z[x]$ may be irreducible, yet fail to be prime infinitely often.
For example, $f(x) = x^2+x +2$ is irreducible, but $f(n)$ is divisible by $2$ for all $n \in \Z$.
Similarly, $f(x) = x^3-x+3$ is irreducible and
\begin{equation*}
x^3 -x + 3 \equiv x^3 - x \equiv x(x-1)(x+1) \equiv 0 \pmod{3},
\end{equation*}
so $f(n)$ is divisible by $3$ for all $n \in \Z$. 
Thus, we must insist that $f$ does not vanish identically modulo any prime.
\end{itemize}

\subsection{Effect of the degree.}
Suppose that $f \in \Z[x]$ is nonconstant, irreducible, and does not vanish identically modulo any prime.
Let $d = \deg f$ and suppose that $f$ has a positive leading
coefficient, $c$.  Then $f(x) \sim cx^d$ and our heuristic from
the prime number theorem suggests that the probability that  $f(n)$ is prime is about
\begin{equation}\label{eq:OneLog}
\frac{1}{\log f(n)} \sim \frac{1}{\log (cx^d)} = \frac{1}{d \log x + \log c} \sim \frac{1}{d \log x}.
\end{equation}
This suggests that 
\begin{equation}\label{eq:WrongQ}
Q(f;x)
\ \sim \ \sum_{n=2}^{\lfloor x \rfloor} \frac{1}{d \log x} 
\ \sim \ 
\frac{1}{\deg f} \int_2^x \frac{dt}{\log t}.
\end{equation}
Is this correct?  We should do some computations to see whether this pans out.

\subsection{A sanity check}

Consider the polynomial 
\begin{equation*}
f(x) = x^2+1,
\end{equation*}
which is nonconstant, irreducible, and has a positive leading coefficient.
Since $f(0) = 1$, it follows that $f$ does not vanish identically modulo any prime.
Landau's conjecture is that $f$ assumes infinitely many prime values; that is,
$Q(f;x) \to \infty$.

According to \eqref{eq:WrongQ}
\begin{equation}\label{eq:Wrong}
Q(f;N) \ \sim \  \frac{1}{2} \Li(N).
\end{equation}
However, the numerical evidence disagrees; see Table \ref{Table:Disagree}.
On the positive side, the loss is not total since it appears that our estimate is only
off by a constant factor.  What is this constant factor and where does it come from?

\begin{table}
\begin{equation*}
\begin{array}{c|ccc}
N & Q(f;N) & \frac{1}{2}\Li(N) & Q(f;N) / \frac{1}{2}\Li(N) \\
\hline
100  & 19 &  15 & 1.3067\\
1{,}000 & 112 & 88  & 1.26866 \\
10{,}000 & 841 & 623 & 1.3509\\
100{,}000 &  6{,}656 & 4{,}814 & 1.38252\\
1{,}000{,}000 & 54{,}110 & 39{,}313 & 1.37638\\
10{,}000{,}000 & 456{,}362 & 332{,}459 &  1.37269\\
100{,}000{,}000 & 3{,}954{,}181 & 2{,}881{,}104  & 1.37245\\
1{,}000{,}000{,}000 & 34{,}900{,}213 & 25{,}424{,}617 & 1.37269\\
\end{array}
\end{equation*}
\caption{The estimate \eqref{eq:Wrong} is clearly incorrect.  However, the ratio
between the correct answer and our prediction appears to converge slowly to a constant
(the value of $\frac{1}{2}\Li(N)$ is rounded to the nearest integer for display purposes).}
\label{Table:Disagree}
\end{table}

\subsection{Making a correction}
We were too quick to celebrate the fact that $f$ does not vanish identically modulo any prime.  
Our prediction needs to take into account how likely it is that $f(n) \equiv 0 \pmod{p}$.
For example, $f(n) \equiv n+1 \pmod{2}$ and hence $f(n)$ is even with probability $\frac{1}{2}$.

If we assume for the sake of our heuristic argument that divisibility by distinct primes $p$ and $q$
are independent events, then we should weight our prediction by\footnote{An important convention we adhere to throughout this paper is that 
the letter $p$ always denotes a prime number.
A product or sum indexed by $p$ indicates that that product or sum runs over all prime numbers.}
\begin{equation}\label{eq:BadProduct}
\prod_p \left( 1 - \frac{\omega_f(p)}{p} \right),
\end{equation}
in which $\omega_f(p)$ is the number of solutions to $f(x) \equiv 0 \pmod{p}$.  

However, there is a problem.
The constant factor suggested by Table \ref{Table:Disagree}, approximately $1.372$, is greater than one, whereas \eqref{eq:BadProduct}
is not.
Therefore, the preceding analysis cannot be correct.
More seriously, there are convergence issues with \eqref{eq:BadProduct}; 
see Section \ref{Section:Product} for information about infinite products.

We need to weight the factors in \eqref{eq:BadProduct}
against the probabilities that randomly selected integers are not divisible by $p$.
This suggests that we adjust \eqref{eq:Wrong} by
\begin{equation}\label{eq:CfLandau}
C(f) \ = \ \prod_p \left( 1 - \frac{1}{p} \right)^{-1}
\left(1- \frac{\omega_f(p)}{p} \right) \ = \
\prod_p \frac{p - \omega_f(p)}{p-1}.
\end{equation}
Does this agree with our numerical computations?
To compute $\omega_f(p)$, we need to count the number of solutions to
$x^2 +1 \equiv 0 \pmod{p}$.
Since $-1$ is a square modulo $p$ if and only if $p = 2$ or $p \equiv 1 \pmod{4}$ \cite{Niven}, 
\begin{equation*}
\omega_f(p) = 
\begin{cases}
1 & \text{if $p=2$},\\
2 & \text{if $p \equiv 1 \pmod{4}$},\\
0 & \text{if $p \equiv 3 \pmod{4}$}.
\end{cases}
\end{equation*}
The hundred millionth partial product of \eqref{eq:CfLandau} yields 
\begin{equation*}
C(f)\  \approx \ 1.37281,
\end{equation*}
which agrees with the data in Table \ref{Table:Disagree}.
In particular, this suggests
an affirmative answer to Landau's problem.

Let us pause to summarize the discussion so far.
For a single polynomial $f$, we suspect that
\begin{equation*}
Q(f;x)\  \sim \ \frac{C(f)}{\deg f} \int_2^x \frac{dt}{\log t},
\end{equation*}
in which
\begin{equation}\label{eq:C1Var}
C(f) \ = \ \prod_p \left( 1 - \frac{1}{p} \right)^{-1}
\left(1- \frac{\omega_f(p)}{p} \right).
\end{equation}
This is the Bateman--Horn conjecture for a single polynomial.
What about families of multiple polynomials?

\subsection{More than one polynomial}
Suppose that $f_1,f_2,\ldots,f_k \in \Z[x]$ are distinct irreducible polynomials with positive leading coefficients.
The same reasoning in \eqref{eq:OneLog} tells us the probability that 
all of the $f_i(n)$ are prime is
\begin{equation*}
\prod_{i=1}^k \frac{1}{\log f_i(n)}\ \sim \ \prod_{i=1}^k \frac{1}{d_i \log n} \ = \ \frac{1}{(\prod_{i=1}^k \deg f_i) (\log n)^k}. 
\end{equation*}
Thus, the expected number of $n$ at most $x$ for which $f_1(n), f_2(n),\ldots,f_k(n)$ are simultaneously 
prime is around
\begin{equation*}
\int_2^x \frac{1}{(\prod_{i=1}^k \deg f_i) (\log n)^k} \ = \ \frac{1}{\prod_{i=1}^k \deg f_i} \int_2^x \frac{dt}{(\log t)^k}.
\end{equation*}
As before, we must amend this with a suitable correction factor.

Although perhaps no single $f_i$ vanishes identically modulo a prime,
these polynomials might conspire to make 
\begin{equation}\label{eq:ffffk}
f = f_1 f_2 \cdots f_k
\end{equation}
vanish identically modulo some prime.
For example, neither $f_1(x) = x$ nor $f_2(x) = x-1$ vanish identically modulo a prime, although their
product $f(x) = x(x-1)$ vanishes identically modulo $2$.  This ``congruence obstruction'' prevents
$n$ and $n+1$ from being simultaneously prime infinitely often.
Consequently, we must require that $f$ does not vanish identically modulo any prime.

With $f$ as in \eqref{eq:ffffk}, one final adjustment to \eqref{eq:C1Var} is necessary.  
Instead of dividing by $1 - 1/p$ in \eqref{eq:C1Var}, we must now divide by $(1-1/p)^k$, the probability that
a randomly selected $k$-tuple of integers has no element divisible by $p$.

\subsection{The Bateman--Horn conjecture}
The preceding heuristic deductions make a compelling argument in favor of the following conjecture.

\bigskip\noindent\fbox{\begin{minipage}{0.975\textwidth}
\medskip\noindent\textbf{Bateman--Horn Conjecture.}
\emph{Let $f_1,f_2,\ldots,f_k \in \Z[x]$ be distinct irreducible polynomials with positive leading coefficients, and let
\begin{equation}\label{eq:Q}
Q(f_1,f_2,\ldots,f_k;x) \ =\  \#\{ n \leq x : \text{$f_1(n),f_2(n),\ldots,f_k(n)$ are prime}\}.
\end{equation}
Suppose that $f = f_1f_2\cdots f_k$ does not vanish identically modulo any prime.
Then
\begin{equation}\label{eq:BH}
Q(f_1,f_2,\ldots,f_k;x)\ \sim\  \frac{C(f_1,f_2,\ldots,f_k) }{\prod_{i=1}^k \deg f_i} 
\int _{2}^{x}\frac{dt}{(\log t)^k},
\end{equation}
in which 
\begin{equation}\label{eq:C}
C(f_1,f_2,\ldots,f_k) \ = \ \prod_p \left( 1 - \frac{1}{p} \right)^{-k}
\left(1- \frac{\omega_f(p)}{p} \right)
\end{equation}
and $\omega_f(p)$ is the number of solutions to $f(x) \equiv 0 \pmod{p}$.
}
\end{minipage}}
\bigskip

Under the hypotheses of the Bateman--Horn conjecture,
the infinite product \eqref{eq:C} always converges.  However, the proof is
delicate and nontrivial; see Section \ref{Section:Converge} for the details. 

The only case of the Bateman--Horn conjecture that has been proven is the
prime number theorem for arithmetic progressions (Theorem \ref{Theorem:PNTAP}). 
However, an upper bound similar to \eqref{eq:BH} is known to be true.
The Brun sieve provides a 
constant $B$ that depends only on $k$ and the degrees of the polynomials involved such that
\begin{equation*}
Q(f_1,f_2,\ldots,f_k;x) \, \leq \, \frac{B\,C(f_1,f_2,\ldots,f_k) }{\prod_{i=1}^k \deg f_i} \int_2^x \frac{dt}{(\log t)^k}
\end{equation*}
for sufficiently large $x$ \cite[Thm.~3, Sect.~I.4.2]{Tenenbaum}.  Thus, the prediction afforded by the Bateman--Horn
conjecture is not unreasonably large.

\section{Historical background}\label{Section:History}

Before proceeding to applications and examples of the Bateman--Horn conjecture,
we first discuss its historical context.  In particular, we briefly examine several important antecedents
that the conjecture generalizes.  We are  fortunate to have available the personal recollections
of Roger A.~Horn, who was kind enough to provide his account of the events leading up to the formulation
of the conjecture.

\subsection{Predecessors of the conjecture}\label{Section:Prehistory}

The Bateman--Horn conjecture is the culmination of hundreds of years of theorems
and conjectures about the large-scale distribution of the prime numbers \cite{Guy}.
In Section \ref{Section:Heuristic} we arrived at the conjecture from the prime number theorem
and heuristic reasoning based upon the Cram\'er probabilistic model of the primes (see \cite{granville}
for a nice exposition of this model). Although this is easy to do in hindsight, in reality the Bateman--Horn 
conjecture evolved naturally from a family of interrelated conjectures, all of which remain open.
We state these conjectures in modern terminology and with our present notation for the sake
of uniformity and clarity.

\medskip\noindent\textbf{Bunyakovsky conjecture \cite{Bouniakowsky} (1854)}:
\emph{Suppose that $f \in \Z[x]$ is irreducible, $\deg f \geq 1$, the leading coefficient of $f$ is positive,
and the sequence $f(1),f(2),\ldots$ is relatively prime.  Then $f(n)$ is prime infinitely often.}
\medskip

This conjecture, which concerns prime values assumed by a single polynomial,
was proposed by Viktor Yakovlevich Bunyakovsky (1804--1889).  It implies, for example,
Landau's conjecture on the infinitude of primes of the form $n^2+1$.
The condition that $f(1),f(2),\ldots$ is relatively prime is equivalent to the assumption that
$f$ does not vanish identically modulo any prime, which appears in the Bateman--Horn conjecture.
Dirichlet's theorem on primes in arithmetic progressions (1837) is the degree-one case of the Bunyakovsky conjecture.

\medskip\noindent\textbf{Dickson's conjecture \cite{Dickson} (1904)}:
\emph{If $f_1,f_2,\ldots,f_k \in \Z[x]$ are of the form $f_i(x) = a_i x + b_i$, 
in which each $a_i$ is positive, and there is no congruence obstruction,
then $f_1(n), f_2(n),\ldots,f_k(n)$ are simultaneously prime infinitely often.}
\medskip

This was conjectured by Leonard Eugene Dickson (1874--1954)
as an extension of Dirichlet's theorem.
By a ``congruence obstruction'' we mean that the $f_1,f_2,\ldots,f_k$ are not prevented
from assuming infinitely many prime values by some combination of congruences.  For example,
$f_1(x) = x+3$, $f_2(x) = x+7$, and $f_3(x) = x-1$ are congruent modulo $3$ to
$x$, $x+1$, and $x+2$, respectively.  For each $n \in \N$,
at least one of $f_1(n), f_2(n), f_3(n)$ is divisible by three.  Since these polynomials
are nonconstant, this prevents them from being simultaneously prime infinitely often.

\medskip\noindent\textbf{First Hardy--Littlewood Conjecture \cite{Hardy} (1923)}:
\emph{
Let $0 < m_1 < m_2 < \cdots < m_k$.  Unless there is a congruence obstruction,
the number of primes $q\leq x$ such that $q+2m_1, q+2m_2,\ldots,q+2m_k$ are prime is
asymptotic to
\begin{equation*}
2^k \prod_{\text{$p$ odd}} \left( 1 - \frac{1}{p} \right)^{-(k+1)}\left(1 - \frac{w(p;m_1,m_2,\ldots,m_k)}{p} \right) \int_2^x \frac{dt}{(\log t)^{k+1}},
\end{equation*}
in which $w(p;m_1,m_2,\ldots,m_k)$ is the number of distinct residues
of $0,m_1,m_2,\ldots,m_k$ modulo $p$.}
\medskip

Unlike the conjectures of Bunyakovsky and Dickson, the first Hardy--Littlewood conjecture
provides an asymptotic expression for the number of primes of a given form.  
It is a special case of the Bateman--Horn conjecture with
\begin{equation*}
f_1(x) = x, \quad  f_2(x) = x+2m_1, \ldots, \quad f_{k+1}(x) = x + 2m_k.
\end{equation*}
There are $k+1$ polynomials involved, which accounts for the power $k+1$ that appears
in the product and the integrand.

The classic paper \cite{Hardy} of Hardy and Littlewood is full of conjectures, labeled ``Conjecture A'' through ``Conjecture P.''
Most of these are subsumed under what is now known as the First Hardy--Littlewood conjecture, which we have just stated.
Hardy and Littlewood end their paper with the remark:
\begin{quote}\small
We trust that it will not be supposed that we attach any exaggerated importance to the speculations which we 
have set out in this last section. We have not forgotten that in pure mathematics, and in the Theory of Numbers in particular, 
`it is only proof that counts'. It is quite possible, in the light of the history of the subject, that the whole of 
our speculations may be ill-founded. Such evidence as there is points, for what it is worth, in the 
opposite direction. In any case it may be useful that, finding ourselves in possession of an apparently fruitful method, 
we should develop some of its consequences to the full, even where accurate investigation is beyond our powers.
\end{quote}
At least one of their conjectures is ``ill-founded.''
The second Hardy--Littlewood conjecture asserts that $\pi(x+y) \leq \pi(x) + \pi(y)$ for $x \geq 2$.
In 1974, Douglas Hensley and Ian Richards proved that the second conjecture is incompatible with the first;
see \cite{Hensley1}, as well as \cite{Hensley2,Richards}.
It is not known which of the two conjectures is true, although most number theorists favor the first.

\medskip\noindent\textbf{Schinzel's Hypothesis H \cite{Schinzel} (1958)}: 
\emph{Let $f_1,f_2,\ldots,f_k$ be distinct irreducible, integer-valued polynomials that have
positive leading coefficients. 
If for each prime $p$ there exists an $m \in \N$ such that none of the values $f_1(m),f_2(m),\ldots,f_k(m)$ are divisible by $p$, then 
there are infinitely many $n \in \N$ such that $f_1(n),f_2(n),\ldots,f_k(n)$ are prime.}
\medskip

This general qualitative predecessor of the Bateman--Horn conjecture
was formulated by Andrzej Schinzel (1937--) in 1958. At that time, Schinzel was a student of Wac\l{}aw Sierpi\'nski (1882--1969) at Warsaw University, and the hypothesis was first stated in its general form in their joint paper~\cite{Schinzel}. In fact, Schinzel wrote the reviews of the two papers
of Bateman and Horn \cite{Bateman, BatemanHorn-2} for Mathematical Reviews and Zentralblatt MATH, the main reviewing services of the American and European Mathematical Societies, respectively. 

The Bateman--Horn conjecture is a quantitative version of hypothesis H.
The hypotheses of both conjectures are essentially the same.
The condition that for each prime $p$ there exists an integer $m$ 
such that none of the values $f_1(m),f_2(m),\ldots,f_k(m)$ are divisible by $p$ is equivalent to the condition that the product 
$f_1f_2 \cdots f_k$ does not vanish identically modulo any prime. 

The Bateman--Horn conjecture unifies all of the conjectures above in one bold prediction.
It provides an asymptotic expression for the relevant counting function
and, moreover, its predictions agree well with numerical computation.
We will chronicle many consequences of the conjecture in Sections \ref{Section:Single} and \ref{Section:Multiple}.

\subsection{Bateman, Horn, and the ILLIAC}

Paul T. Bateman (1919--2012) earned his Ph.D.~in 1946 under Hans Rademacher 
(1892--1969) at the University of
Pennsylvania.  He joined the mathematics department of the University of Illinois, Urbana-Champaign
in 1950 and stayed there until his retirement in 1989, after which he was Professor Emeritus.
He was department head from 1965 until 1980 and is credited by many for his leadership,
incredible memory, and work ethic.
Harold G.~Diamond \cite{Diamond} tells us 
\begin{quote}\small
Paul is perhaps best known to the number theory community for the Bateman--Horn conjectural asymptotic formula for the number of 
$k$-tuples of primes generated by systems of polynomials\ldots
Their formula extended and quantified several famous conjectures of Hardy and J.E.~Littlewood, and of Andrzej Schinzel, 
and they illustrated its quality with calculations. This topic has been treated in dozens of subsequent papers.
\end{quote}
Hugh Montgomery adds
\begin{quote}\small
Bateman not only organized an active number theory group in Urbana, with such people as John Selfridge, Walter Philipp,
Harold Diamond, and Heini Halberstam, but he also did a lot to promote number theory around the country, and also he did a huge amount of service to the AMS.
Later, when Batemen died, he didn't get all the honor and credit he deserved. He had lived so long, that the (comparatively young) editor of the AMS Notices had
no idea who Bateman had been. He insisted on just a very short (1 page or so) obituary, so many of the reminiscences never saw the light of day. Harold Diamond
may still have drafts of what we wanted to publish.  \cite{MontgomeryPersonal}
\end{quote}
Fortunately, it appears that Diamond was able to publish much of the desired memorial tribute online \cite{Diamond}.

\begin{figure}
    \centering
    \begin{subfigure}[t]{0.525\textwidth}
        \centering
        \includegraphics[height=2.25in]{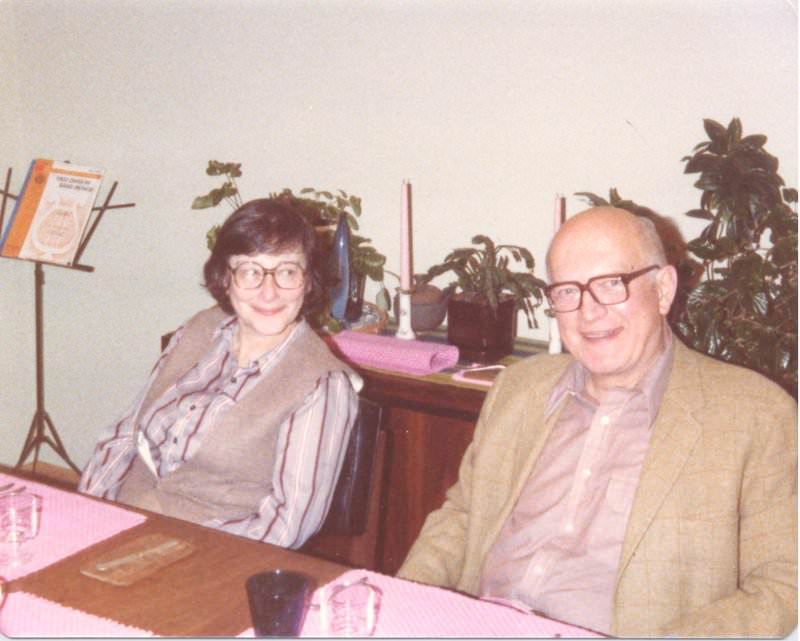}
        \caption{Paul and Felice Bateman in January 1980. 
        Photo courtesy of Harold G.~Diamond.}
    \end{subfigure}
    \quad
    \begin{subfigure}[t]{0.425\textwidth}
        \centering
        \includegraphics[height=2.25in]{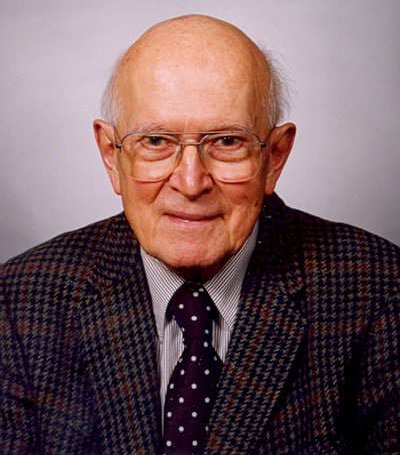}
        \caption{Photo courtesy of the University of Illinois mathematics department}
    \end{subfigure}
    \caption{Paul T.~Bateman (1919--2012)}
    \label{Figure:Bateman}
\end{figure}

Roger A.~Horn (1942--) received his Ph.D. from Stanford in 1967, under the direction of 
Donald Spencer (1912--2001) and Charles Loewner (1893--1968);
see Figure \ref{Figure:Horn}.  He worked briefly at Santa Clara University 
before moving to Johns Hopkins in 1968, where he founded the Department of Mathematical Sciences in 1972.
He remained at Johns Hopkins until 1992, when he moved to the University of Utah as Research Professor.
He retired in 2015 and currently resides in Tampa.

Horn is known best for his long and storied career in matrix analysis.
Among his chief publications are the
classic texts \emph{Matrix Analysis} \cite{MA} and \emph{Topics in Matrix Analysis} \cite{Topics}, both coauthored with Charles Johnson.
Of his many papers, only two are on number theory;
both of these date from the early 1960s and concern the Bateman--Horn conjecture \cite{Bateman, BatemanHorn-2}.
Consequently, many of his close colleagues are unaware of his connection to a famous
conjecture in number theory.\footnote{A common misconception is that Roger Horn is the ``Horn'' from the famed Horn conjecture about the eigenvalues of a sum of two Hermitian matrices, settled in 1999 by Knutson--Tao \cite{Knutson} and Klyachko \cite{Klyachko}.
That distinction belongs to Alfred Horn (1918--2001), who made the conjecture in 1962 \cite{HornWrong}, the
same year in which the Bateman--Horn conjecture appeared \cite{Bateman}.}
For example,
the third named author wrote a linear algebra textbook \cite{GarciaHorn}
with Roger Horn before he learned, in the course of 
a number theory project \cite{GKL, GLS},
that Roger was ``the'' Horn from Bateman--Horn!

\begin{figure}
    \centering
    \begin{subfigure}[t]{0.5195\textwidth}
        \centering
        \includegraphics[width=\textwidth]{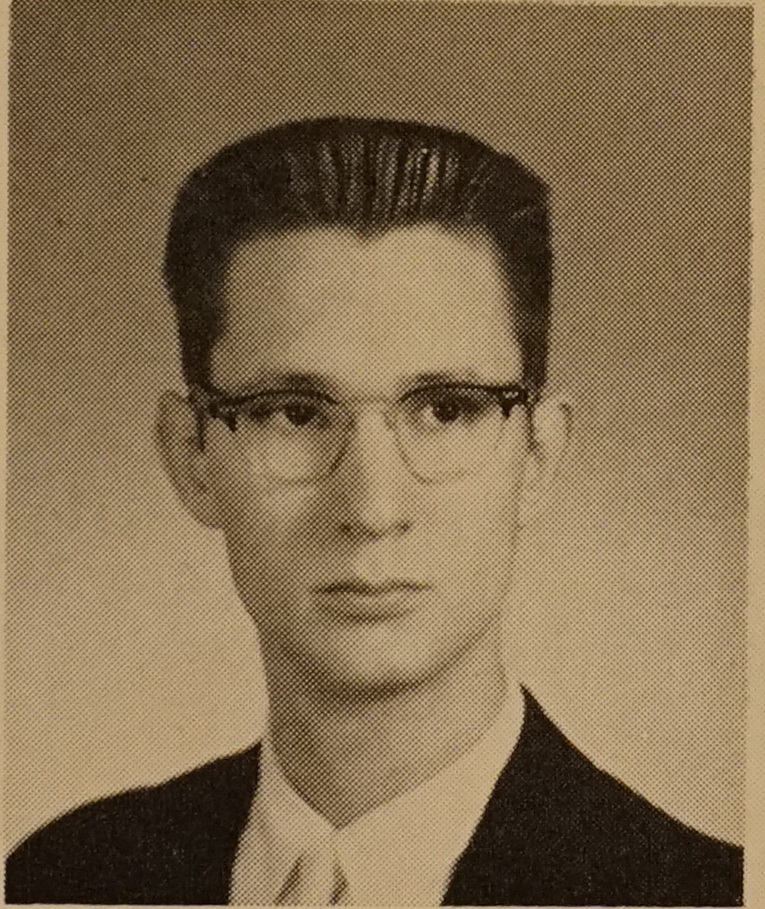}
        \caption{Roger Horn in his 1963 Cornell graduation photo, around the time
        of his work under Bateman.}
    \end{subfigure}
    \begin{subfigure}[t]{0.4105\textwidth}
        \centering
        \includegraphics[width=\textwidth]{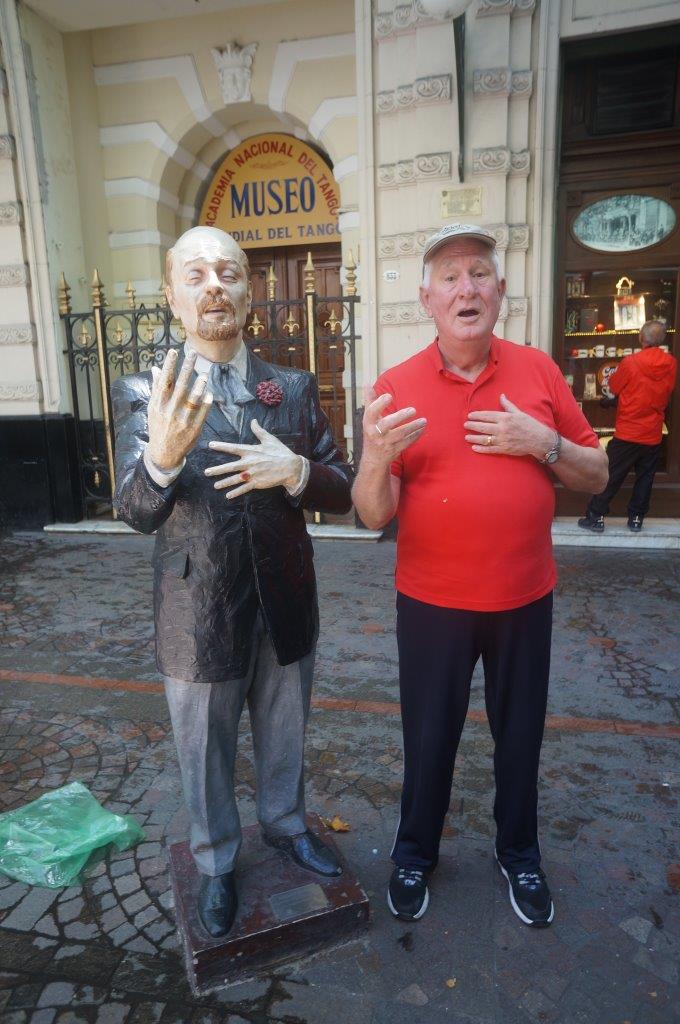}
        \caption{Roger Horn on vacation in Buenos Aires, January 2017. }
    \end{subfigure}

    \caption{Photographs provided courtesy of Roger A.~Horn.}
    \label{Figure:Horn}
\end{figure}

How did Roger Horn co-propose an important conjecture in a field so far from his own?
We are fortunate to have access to his detailed recollections \cite{HornPersonal}.
\begin{quote}\small
In the early 1960s, the National Science Foundation funded several summer programs intended to introduce college mathematics students to computing. In 1962 I applied to, and was accepted into, one of those programs, which was hosted by the Computing Center at the University of Illinois in 
Urbana-Champaign.  

There were about 10 participants, from all over the country. We were housed in 
university dorm rooms, attended classes in the Computing Center, and had 
unlimited access to the hottest computer on campus, the ILLIAC, which later 
was known as the ILLIAC I when its successor, the ILLIAC II was built. 
\end{quote}

The ILLIAC (\textbf{Ill}inois \textbf{A}utomatic \textbf{C}omputer), which powered up 
on September 22, 1952, was the first computer to be built and owned by a United States 
academic institution; see Figure \ref{Figure:ILLIAC}.  It was the second of two identical computers,
the first of which was the ORDVAC (Ordnance Discrete Variable Automatic Computer), built by the 
University of Illinois for the government's Ballistics Research Laboratory.  
The two machines employed the architecture proposed by 
John von Neumann in 1945.

\begin{figure}
    \centering
        \includegraphics[width=0.9\textwidth]{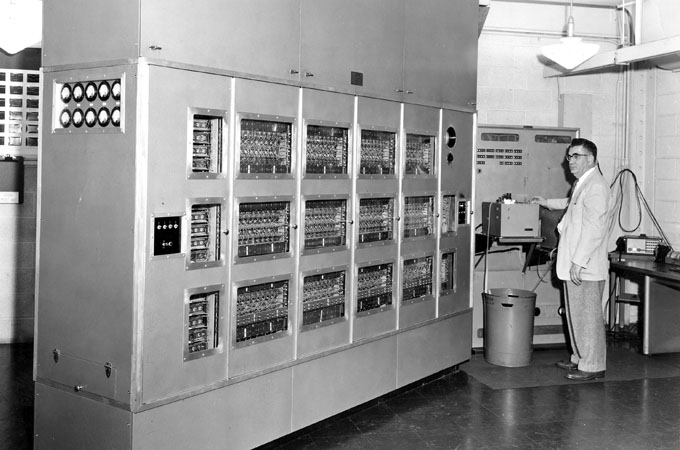}
    \caption{The ILLIAC I computer around 1952.
    Courtesy of the University of Illinois Archives}
    \label{Figure:ILLIAC}
\end{figure}

\begin{quote}\small
In those days, 
universities built their own computers: IBM hardware was of the punch-card variety, 
for which businesses were the primary customers; they were not well suited for scientific work.
It was the size of a small house and it consumed a prodigious amount of electric power. It stopped working frequently when one of its thousands of vacuum tubes died. We programmed it in hexadecimal machine code; no high-level user language (BASIC or FORTRAN, for example) was ever written for it.
\end{quote}

According to the archives of the University of Illinois, the ILLIAC
weighed two tons, measured $10 \times 2 \times 8.5$ feet,\footnote{Horn says ``this was only the console, the big box visible in Figure \ref{Figure:ILLIAC}.
All this stuff, and a huge power supply, was in a big adjoining room (the size of a small house).''} 
and required approximately $2{,}800$
vacuum tubes to operate \cite{Illinois}.  
A later survey, published in 1961 and based upon observations made in 1959,
gives quite different figures:  4{,}427 vacuum tubes of twenty-seven different types \cite{Weik}.
It is likely that the system was somewhat expanded and upgraded in the period since its construction 
in 1951 and the latter figure perhaps more closely approximates the system that Bateman and Horn used.

\begin{quote}\small
Some classes were organized for us. One was on Boolean logic and circuit design, taught by one of the engineers who was working on the design of ILLIAC II. Another was on numerical analysis, taught by Herb Wilf, who was a young assistant professor and author (with Anthony Ralston) of a new numerical analysis textbook
[Mathematical Methods for Digital Computers]. I first learned about interpolation and orthogonal polynomials in that class.

Initially, we were given small problems to program for the ILLIAC to develop our programming skills. Programs were typed onto paper tape with the same Teletype machines used by Western Union. Noisy! We submitted our tapes to the ILLIAC operator, who fed them into the machine. We did that a lot because most of the time our programs crashed. The ILLIAC had a small speaker hooked up to a bit in its accumulator register, and gave out a high-pitched whine when it went into a loop. The operator had to flip a ``kill'' switch to stop it, and that was embarrassing. 
\end{quote}

The ILLIAC could read punched paper tape at a rate of 300 characters per second.  Moreover,
``five hole teletype tape is used. Numerical data are read with a $4$-hole code.
Alphanumerical data employs a $5$-hole code and a special instruction'' \cite{Weik}.
Output appeared on paper tape at 60 characters per second, or on a page
printer at a sluggish 10 characters per second.  

\begin{quote}\small
After a couple of weeks we started work on some projects. The organizers had lined up some faculty who were willing to mentor us and supervise projects. I chose two: One was as part of a team of three supervised by Herb Wilf. We did a lot of calculations in an attempt to find a counterexample to the P\'olya--Schoenberg conjecture (if two normalized univalent analytic functions on the unit disk have the property that each maps the unit disk onto a convex domain, then their Hadamard product has the same property). Part of the computation required testing some very large Hermitian matrices for positive definiteness, so I learned something about that topic. All of our runs produced negative results\ldots no counterexamples found. This was a good thing, because about 10 years later the conjecture became a theorem.
\end{quote}

The conjecture, stated in 1958  by George P\'olya (1887--1985)
and Isaac Schoenberg (1903--1990) \cite{Polya}, became a theorem in 1973 when it was proved by 
Stephan Ruscheweyh and Terence Sheil-Small \cite{Ruscheweyh}.  Although we do not wish to drift too far afield, there 
are a few tangential remarks of mild historical interest that are worth making.
First, Herbert Saul Wilf (1931--2012) was at Illinois from 1959 to 1962, after which he moved to the University of Pennsylvania.  
Thus, Horn must have worked with Wilf just before his departure.  Wilf's 1963 paper on the P\'olya--Schoenberg conjecture
also mentions Horn's contribution and identifies several other participants of the 1962 summer research program:
\begin{quote}\small
The machine program was planned and executed by Messrs. Roger A. Horn (Cornell University), Forrest R. Miller Jr. (University of Oklahoma) and Gerald Shapiro (Massachusetts Institute of Technology) who visited the Digital Computer Laboratory at Illinois during a summer program for undergraduates in Applied Mathematics sponsored by the National Science Foundation. These calculations were made possible largely by their dedication and enthusiasm. \cite{Wilf}
\end{quote}

Now back to number theory and Roger Horn's account of the origins of the Bateman--Horn conjecture  \cite{HornPersonal}.
\begin{quote}\small
My other project was a lone effort supervised by Paul T. Bateman, 
a famous analytical number theorist; I think he was chair of the math department at the time. 
His Ph.D. advisor was Hans Rademacher. He had me read some papers that dealt with a variety of number-theoretic conjectures (there were then, and still are now, a LOT of them!) with the goal of choosing something that might be amenable to experimental computation. Eventually, we settled on the problem reported on in our 1962 Math.~Comp.~paper. I burned up about 7 hours of ILLIAC time, but the results were very interesting and gave increased confidence in the conjectures. 
\end{quote}

The UIUC mathematics department website and two short biographies of Bateman 
assert that he was department head (not chair) from 1965 until 1980 \cite{Biography, Diamond}.
Hugh Montgomery tells us that ``Bateman was not the chair of the math dept when I arrived as a freshman in 1962.
The chair at that time was M.M.~Day.  But during my sophomore year, Day became
ill with an ulcer, and Bateman was then asked to take over. He was probably chair
first, and then head later'' \cite{MontgomeryPersonal}.\footnote{Montgomery also remarks ``I was interested in number theory already when I was in high school. At Illinois I started taking their honors math courses. I got to know Bateman
during the second half of my sophomore year, when I took his graduate-level problem-solving class. I worked around 40 hrs per week on that one class, while carrying a full load of other courses, but it was worth it. During the summer after my junior year, he
had me stay in Urbana and do a research project, probably on the same grant that
Horn had been on.  It was sort of a precursor of REU.''}

Of greater interest to us are the computations mentioned above.
The paper \cite{Bateman}, in which the Bateman--Horn conjecture is stated, says the following.
\begin{quote}\small 
The second-named author [Roger Horn] used the ILLIAC to prepare a list of the 776 primes of
the form $p^2 + p + 1$ with $p$ a prime less than $113{,}000$. (The program used was a straightforward one, and the running time was about $400$ minutes.) 
The first $209$ of these primes are listed by Bateman and Stemmler who considered primes of the form 
$p^2 + p + 1$ in connection with a problem in algebraic number theory.
\end{quote}

The ``Stemmler'' mentioned above is Rosemarie M.S. Stemmler, a student of Bateman who received her Ph.D. in 1959 \cite{BatemanStemmler}.  
Bateman and Horn
computed $Q(f_1,f_2;x)$ for various $x \leq 113{,}000$ with
$f_1(t)  = t$ and $f_2(t) = t^2+t+1$.
On the third named author's late-2013 iMac, the same computation takes only a tenth of a second!

Although the summer drew to a close, Horn continued to work on the project:
\begin{quote}\small
When the summer was over, I went back to Cornell for my senior year and found that they had taken delivery of a brand new CDC [Control Data Corporation] 1604 computer.
It took a while for folks to discover that it was in operation and move their work to a new programming environment, so I was able to get quite a lot of overnight time on the machine, which was much faster than the ILLIAC and a lot more reliable. It had FORTRAN, too! I ran a lot of additional experiments that were reported in our 1965 Symposia in Pure Math VIII paper \cite{BatemanHorn-2}.
And then I graduated, went to graduate school, took other directions in my research, and haven't thought about these number theory issues since 1963.
\end{quote}

We wrap things up with a humorous anecdote connected to the Bateman--Horn conjecture.
Serge Lang (1927--2005), in his book \emph{Math Talks for Undergraduates} provides one of 
the few expositions of the conjecture \cite{Lang}.  
In the introduction, he claims that his tone was too conversational and informal for certain editors:
\begin{quote}\small
[Paul] Halmos once characterized this style as ``vulgar'', and obstructed publication of excerpts in the \emph{Math Monthly}.
  A decade later, in the 1990s, the present talk was offered for publication again in the \emph{Math Monthly}, and was turned down by the editor (Roger Horn, 
this time) because of the spoken style. 
Well, I like the spoken style, and I find it effective.  Go figure.  \cite[p.1]{Lang}
\end{quote}

There is a remarkable confluence here.
Paul Halmos, the academic grandfather of the third named author, 
was editor(-in-chief) of the \emph{American Mathematical Monthly} from 1982 to 1986.
Herbert S. Wilf, who we met above in connection to the P\'olya--Schoenberg conjecture, was the editor from 1987-1991.
Roger Horn was editor from 1997 to 2001!

Horn recalls that he ``had a memorable bad experience once with Lang, while I was Editor of the Monthly.''
Although he has no recollection of a submission related to the Bateman--Horn conjecture, he does remember several submissions on other topics.
He also vividly remembers a phone call in which
``[Lang] shouted at me for ten minutes or so, and then hung up.''

\section{Why does the product converge?}
\label{Section:Converge}

We now discuss the convergence of the product \eqref{eq:C} that defines the Bateman--Horn constant 
$C(f_1,f_2,\ldots,f_k)$.  This is a delicate argument that requires elements of both algebraic and analytic number theory,
along with a few tricks to deal with conditionally convergent infinite products. 
In \cite[p.~36]{luca}, the authors state:
\begin{quote} \small
It is not even clear that in formula (2.18) the expression $C(f_1,f_2,\ldots,f_k)$ represents a product which converges to a positive limit.
\end{quote}

\noindent
We wish to provide a thorough account here since most of these details are suppressed in the original source \cite{Bateman}.  

\subsection{Infinite products}\label{Section:Product}

Before we can proceed with the proof that the product \eqref{eq:C} that defines the Bateman--Horn constant converges,
we require a few general words about infinite products.  

The only way that a zero factor can appear in the evaluation of $C(f_1,f_2,\ldots,f_k)$
is if $\omega_f(p) = p$ for some prime $p$; that is, if $f$ vanishes identically modulo $p$.
This is prohibited by the hypotheses of the Bateman--Horn conjecture,
so we can safely ignore this possibility.  
Let $a_n$ be a sequence in $\mathbb{C} \backslash \{-1\}$.
Fix a branch of $\log z$ the logarithm with $\log 1 =0$ and for which $\log (1+a_n)$ is defined.
\begin{itemize}
\item We say that $\prod_{n=1}^{\infty} (1+a_n)$ \emph{converges} to $L \neq 0$
if and only if $\sum_{n=1}^{\infty} \log (1+a_n)$ converges to $\log L$.  
Otherwise the infinite product \emph{diverges}.

\item If $a_n$ is a sequence of real numbers and
$\sum_{n=1}^{\infty} \log (1+a_n)$ diverges to $-\infty$, 
then we say that $\prod_{n=1}^{\infty} (1+a_n)$ \emph{diverges to zero}.
In particular, this means that the partial products $\prod_{n=1}^{N} (1+a_n)$ tend to zero as $N \to \infty$.
\end{itemize}

It turns out that the infinite products that arise in the Bateman--Horn conjecture
are often rather finicky.  To handle them, 
we require the following convergence criterion.  Although it is well known in analysis circles
as a folk theorem, we are unable to find a reference that contains a proof.
For the sake of completeness, we provide the proof below.

\begin{lemma}\label{Lemma:Trick}
Let $a_n$ be a sequence in $\mathbb{C} \backslash\{-1\}$.
If $\ \sum_{n=1}^{\infty} |a_n|^2 < \infty$, then 
$\sum_{n=1}^{\infty} a_n$ and $\prod_{n=1}^{\infty}(1+a_n)$ converge or diverge together.
\end{lemma}

\begin{proof}
For $|z| \leq \frac{1}{2}$, 
\begin{equation}\label{eq:Lazy}
\log(1+z) = \sum_{n=1}^{\infty}\frac{(-1)^{n-1} z^n}{n}
= z + \left(-\frac{1}{2} + \frac{z}{3} - \frac{z^2}{4} + \cdots\right)z^2 
= z + z^2 L(z),
\end{equation}
in which 
\begin{equation*}
|L(z)|  \leq \sum_{n=0}^{\infty} \frac{1}{(n+2)2^n} = -2 + \log 16 = 0.77258\ldots < 1.
\end{equation*}
If $\sum_{n=1}^{\infty} |a_n|^2 < \infty$, then there is an $N$ such that $|a_n| \leq \frac{1}{2}$ for $n\geq N$.  Therefore,
\begin{equation*}
\sum_{n=N}^{\infty} \log(1+a_n) 
= \sum_{n=N}^{\infty} a_n + \sum_{n=N}^{\infty} a_n^2 L(a_n),
\end{equation*}
in which the second series on the right-hand side converges absolutely
by the comparison test.  Thus,
\begin{equation*}
\text{$\sum_{n=1}^{\infty} a_n$ converges}
\iff
\text{$\sum_{n=1}^{\infty} \log(1+a_n)$ converges}
\iff
\text{$\prod_{n=1}^{\infty} (1+a_n)$ converges}. \qedhere
\end{equation*}
\end{proof}

\begin{example}\label{Example:Weird}
The hypothesis $\sum_{n=1}^{\infty} |a_n|^2 < \infty$ is necessary in Lemma \ref{Lemma:Trick}.
If
\begin{equation*}
a_n = \frac{(-1)^n}{\sqrt{n \log n}}
\end{equation*}
for $n\geq 2$,
then 
\begin{equation}\label{eq:LCT}
\sum_{n=2}^{\infty} |a_n|^2 = \sum_{n=2}^{\infty} \frac{1}{n \log n}
\end{equation}
diverges by the integral test.  However, $\sum_{n=2}^{\infty} a_n$ converges by the alternating series test while
the second series on the right-hand side of
\begin{equation*}
\sum_{n=4}^{\infty} \log(1+ a_n) = \sum_{n=4}^{\infty} a_n + \sum_{n=4}^{\infty} \frac{L(a_n)}{n \log n}
\end{equation*}
diverges by the limit comparison test against \eqref{eq:LCT}
since $L(a_n) \to -\frac{1}{2}$ by \eqref{eq:Lazy}.\footnote{To use \eqref{eq:LCT} we require $|a_n| \leq \frac{1}{2}$.
Note that $|a_3| > \frac{1}{2}$ and $|a_n| \leq \frac{1}{2}$ for $n \geq 4$.}
\end{example}

The infinite product $\prod_{n=1}^{\infty} (1+a_n)$ \emph{converges absolutely} if
$\prod_{n=1}^{\infty} (1+|a_n|)$ converges; this is equivalent to the convergence of $\sum_{n=1}^{\infty} |a_n|$.
An infinite product that converges but does not converge absolutely is
\emph{conditionally convergent}.

\subsection{Algebraic prerequisites}

Let $\K$ be a number field; that is, a finite algebraic extension of $\Q$.
This implies that each element of $\K$ is algebraic over $\Q$ and that the dimension 
of $\K$ as a $\Q$-vector space is finite.  This dimension is called the \emph{degree} of $\K$ over $\Q$ and denoted by $[\K:\Q]$.

For each $\alpha \in \K$, there is a unique irreducible polynomial $m_{\alpha}(x) \in \Z[x]$ 
with relatively prime coefficients and positive leading coefficient such that $m_{\alpha}(\alpha) = 0$.
This is the \emph{minimal polynomial} of $\alpha$.  
The \emph{degree of $\alpha$}, denoted by $\deg \alpha$, is the degree of the polynomial $m_{\alpha}$, which is at most $[\K:\Q]$.
One can show that 
\begin{equation*}
\O_{\K} := \left\{ \alpha \in \K :  \text{$m_{\alpha}(x)$ is monic} \right\}
\end{equation*}
is a subring of $\K$ (see, for instance, Theorem~2.9 of~\cite{StewartTall} or p.16 of~\cite{Marcus}); it is the \emph{ring of algebraic integers of $\K$}. 
Since $m_n(x) = x-n$ is irreducible for each $n \in \Z$, it follows that $\Z \subseteq \O_{\K}$.

For $\alpha \in \K$, let $\Q(\alpha)$ denote the smallest (with respect to inclusion) subfield of $\K$ that contains
$\Q$ and $\alpha$. 
The following important theorem asserts that every number field is 
generated by a single algebraic integer \cite[Thm.~2.2 \& Cor.~2.12]{StewartTall}.

\begin{theorem}[Primitive Element Theorem]
If $\K$ is a number field, then there is a $\theta \in \O_{\K}$ such that $\K = \Q(\theta)$.
\end{theorem}

If $\K = \Q(\theta)$, then we have the field isomorphism
\begin{equation*}
\K \cong \Q[x] /\langle m_{\theta}(x) \rangle,
\end{equation*}
in which $\langle  m_{\theta}(x) \rangle$ is the (maximal) ideal in $\Q[x]$ generated by 
the irreducible polynomial $m_{\theta}(x)$.  In this case, $[\K:\Q] = \deg \theta$.
Observe that $\Z[\theta]$, the set of integral linear combinations of powers of $\theta$,
is a subring of $\O_{\K}$ and hence $\O_{\K}$ is a ring extension of $\Z[\theta]$. 
The index of $\Z[\theta]$ inside $\O_{\K}$ (as abelian groups), which is finite,
is denoted $[\O_{\K} : \Z[\theta]]$.

We say that $p$ is a \emph{rational prime} if it is a prime in the ring $\Z$;
that is, if $p$ is prime in the traditional sense.
For each rational prime $p$, the set $p\O_{\K}$ is an ideal in $\O_{\K}$.
Although this ideal might not be a prime ideal in $\O_{\K}$, it can be factored
as a product of prime ideals \cite[Thm.~5.6]{StewartTall}.  
Thus, for each rational prime $p$
there exist distinct prime ideals $\PP_1,\PP_2,\ldots,\PP_k \subset \O_{\K}$ 
and positive integers $e_1,e_2,\ldots,e_k$ such that
\begin{equation}\label{eq:IdealFactor}
p\O_{\K} = \PP_1^{e_1} \PP_2^{e_2}\cdots \PP_k^{e_k}.
\end{equation}
This factorization is unique up to permutation of factors.  
Each prime ideal $\PP\subset\O_{\K}$ 
can be present in the factorization for only one rational prime \cite[Thm.~5.14c]{StewartTall}.   

If $e_i > 1$ for some $i$ in \eqref{eq:IdealFactor}, then $p$ \emph{ramifies} in $\K$; 
the exponents $e_1,e_2,\ldots,e_k$ are called \emph{ramification indices}. 
There are only finitely many rational primes $p$ that ramify in a given number field \cite[Cor.~2, p.~73]{Marcus}.
Since prime ideals in $\O_{\K}$ are maximal \cite[Thm.~5.3d]{StewartTall}, it follows that
$\O_{\K}/\PP_i$ is a field for each $\PP_i$ in the factorization \eqref{eq:IdealFactor}.
In fact, it is a finite field of characteristic $p$ \cite[p.~56]{Marcus} and hence
its cardinality is $p^{f_i}$ for some $f_i$, which is called the \emph{inertia degree} of $p$ at $\PP_i$
(the notation $f_i$ is standard and should not be confused with the polynomials in the statement of the Bateman--Horn conjecture).
The \emph{norm} of the ideal $\PP_i$ is 
\begin{equation}\label{eq:NormIdeal}
N(\PP_i) = \left| \O_{\K}/\PP_i \right| = p^{f_i}
\end{equation}
and there are only finitely many prime ideals in $\O_{\K}$ of a given norm \cite[Thm.~5.17c]{StewartTall}.
The factorization \eqref{eq:IdealFactor} is related to the factorization of $m_{\theta}(x)$ modulo $p$.
This connection is given by the Dedekind factorization criterion (see \cite[Prop.~25, p.~27]{Lang}).

\begin{theorem}[Dedekind Factorization Criterion]\label{Theorem:Dedekind} 
Let $\K=\Q(\theta)$, in which $\theta \in \O_{\K}$, and let $p$ a rational prime whose ideal $p\O_{\K}$ factors as in \eqref{eq:IdealFactor}.
If $p \nmid [\O_{\K} : \Z[\theta]]$, then there is a factorization 
\begin{equation*}
m_{\theta}(x) \equiv g_1(x)^{e_1}  g_2(x)^{e_2} \cdots g_k(x)^{e_k} \pmod{p}
\end{equation*}
into powers of irreducible polynomials $g_i(x)$ modulo $p$, in which $\deg g_i(x) = f_i$, 
the inertia degree of $p$ at the corresponding prime ideal $\PP_i$.
\end{theorem}

One immediate and important implication of this theorem is that
\begin{equation*}
\deg \theta = \sum_{i=1}^k e_if_i.
\end{equation*}
Observe also that $m_{\theta}(a) \equiv 0\pmod{p}$ for some $a \in \Z$ if and only if 
$(x-a) \mid m_{\theta}(x)$ modulo $p$.  This occurs if and only if $g_i(x) = x-a$ for some $i$,
in which case $f_i = \deg g_i = 1$ and \eqref{eq:NormIdeal} tells us that the corresponding prime ideal $\PP_i$ 
in the factorization \eqref{eq:IdealFactor} has norm $p$. 
Since there are only finitely many primes that divide the index $[\O_{\K} : \Z[\theta] ]$, we have the following corollary.

\begin{corollary}\label{Corollary:Ideals} Let $g(x) \in \Z[x]$ be a monic irreducible polynomial with root $\theta$ and let $\K = \Q(\theta)$.
For all but finitely many rational primes $p$, the number $\omega_g(p)$ of solutions to 
$g(x) \equiv 0\pmod{p}$ equals the number of prime ideals of norm $p$ in the prime ideal factorization of $p\O_{\K}$.
\end{corollary}

\subsection{Analytic prerequisites}
Later on we will need the following theorem of Leonhard Euler.
We present a proof due to Clarkson \cite{Clarkson};
see \cite{Vanden} for a survey of various proofs.

\begin{theorem}[L.~Euler, 1737]\label{Theorem:Euler}
$\displaystyle\sum_p \frac{1}{p}$ diverges.
\end{theorem}

\begin{proof}
		Let $p_n$ denote the $n$th prime number and suppose toward a contradiction that
		$\sum_{n=1}^{\infty} \frac{1}{p_n}$ converges.  Since the tail end of a
		convergent series tends to zero, let $K$ be so large that
		\begin{equation*}
			\sum_{j=K+1}^{\infty} \frac{1}{p_j}\ <\ \frac{1}{2}.
		\end{equation*}
		Let $Q = p_1 p_2 \cdots p_K$
		and note that none of the numbers
		$$Q+1,\, 2Q+1,\, 3Q+1,\ldots$$
		is divisible by any of the primes $p_1,p_2,\ldots,p_K$.
		Now observe that
		\begin{equation*}
			\sum_{n=1}^N \frac{1}{nQ+1}\ \leq\ \sum_{m=1}^{\infty}
			\bigg(  \sum_{j=K+1}^{\infty} \frac{1}{p_j} \bigg)^m
			\ < \ \sum_{m=1}^{\infty} \left( \frac{1}{2} \right)^m \ = \ 1
		\end{equation*}
		 for $N \geq 1$; the reason for the first inequality is the fact that the sum in the middle, when expanded
		term-by-term, includes every term on the left-hand side
		(and with a coefficient greater than
or equal to $1$).
		This is a contradiction, since 
		$\sum_{n=1}^{\infty} \frac{1}{nQ+1}$ diverges by the integral test.
\end{proof}

A more precise version of the preceding lemma was obtained by 
Franz Mertens (1840--1927). 
Since the proof of Mertens' theorem would draw us too far afield, we refer the reader to 
Terence Tao's exposition for details \cite{TaoMertens}.

\begin{theorem}[Mertens, 1874]\label{Theorem:Mertens}
\begin{equation*}
\sum_{p \leq x} \frac{1}{p} = \log \log x + B + O\left( \frac{1}{\log x} \right).
\end{equation*}
in which $B = 0.2614972128476\ldots$ is the \emph{Meissel--Mertens constant}. 
\end{theorem}

Much of the analytic theory of prime numbers goes through to prime ideals,
\emph{mutatis mutandis}.  Define
\begin{equation*}
\pi_{\K}(x) = \left| \left\{ \PP \subset \O_{\K} : \PP \text{ is a prime ideal and } N(\PP) \leq x \right\} \right|,
\end{equation*}
which is a generalization of the usual prime counting function $\pi(x) = \pi_{\Q}(x)$.  The 
prime number theorem asserts that
$\pi(x) \sim x/\log x$.
This is a special case of Landau's prime ideal theorem \cite{Landau}, \cite[p.~194, p.~267]{MontgomeryVaughan}.

\begin{theorem}[Prime Ideal Theorem]
If $\K$ is a number field, then $\displaystyle \pi_{\K}(x) \sim \Li(x)$.
\end{theorem}

Thus, the asymptotic distribution of prime ideals (by norm) in a number field mirrors that of the prime numbers
in the integers.  Therefore, it is not surprising to find an analogue of Mertens' theorem (Theorem \ref{Theorem:Mertens})
that holds for prime ideals \cite[Lemma 2.4]{Rosen} or \cite[Prop.~2]{Lebacque}.

\begin{theorem}[Mertens theorem for number fields]\label{Theorem:AlgebraicMertens}
If $\K$ is an algebraic number field, then there is a constant $C$ such that
\begin{equation*}
\sum_{N(\PP) \leq x} \frac{1}{N(\PP)} = \log \log x + C + O\left(\frac{1}{\log x} \right),
\end{equation*}
in which the sum runs over all nonzero prime ideals $\PP$ in $\O_{\K}$ of norm at most $x$.
\end{theorem}

We are now in a position to prove the following convergence result
(recall that $p$ always denotes a prime number and that $\sum_p$ means that we sum over all primes).

\begin{lemma}\label{Lemma:Converge}
Let $g(x) \in \Z[x]$ be monic irreducible. For each rational prime $p$, let 
$\omega(p)$ denote the number of solutions to
$g(x) \equiv 0\pmod{p}$.
Then 
\begin{equation*}
\sum_{p} \frac{\omega(p) - 1}{p}
\end{equation*}
converges.
\end{lemma}

\begin{proof}
Let $\K = \Q(\theta)$, in which $\theta$ is a root of $g$.
Then Corollary \ref{Corollary:Ideals} implies that
\begin{equation*}
\sum_{p \leq x} \frac{\omega(p)}{p} = \sum_{N(\PP) \leq x} \frac{1}{N(\PP)} + A,
\end{equation*}
in which the constant $A$ arises from the finitely many rational primes $p$ that are excluded
from Corollary \ref{Corollary:Ideals}.
Theorems \ref{Theorem:Mertens} and \ref{Theorem:AlgebraicMertens} imply that
\begin{align*}
\sum_{p \leq x} \frac{\omega(p) - 1}{p}
&= \sum_{N(\PP) \leq x} \frac{1}{N(\PP)} - \sum_{p\leq x} \frac{1}{p}  + A \\
&= \left[ \log \log x + C + O\left(\frac{1}{\log x} \right) \right] 
- \left[\log \log x + B + O\left(\frac{1}{\log x} \right) \right] + A\\
&= A-B+C + O\left(\frac{1}{\log x} \right) 
\end{align*}
converges to $A-B+C$ as $x \to \infty$.
\end{proof}

\subsection{Convergence of the product}

We are now ready to prove the convergence of the product \eqref{eq:C} that defines
the Bateman--Horn constant. Let
$f_1,f_2,\ldots,f_k \in \Z[x]$ be irreducible
and define $f = f_1f_2\cdots f_k$.  Let
$\omega_i(p)$ and $\omega(p)$ denote the number of solutions in $\Z/p\Z$ to $f_i(x) \equiv 0 \pmod{p}$
and $f(x) \equiv 0 \pmod{p}$, respectively.

\begin{lemma} \label{Lemma:OmegaSplit} 
For all but finitely many primes $p$,
\begin{equation}\label{eq:OmegaSum}
\omega(p) = \omega_1(p) + \cdots + \omega_k(p).
\end{equation}
\end{lemma}

\begin{proof}
Since $p$ is prime, each zero of $f$ in $\Z/p\Z$
is a zero of some $f_i$.  Thus,
\begin{equation*}
\omega(p) \leq \omega_1(p) + \cdots + \omega_k(p).
\end{equation*}
On the other hand, every zero of each $f_i$ in $\Z/p\Z$ is a zero of $f$.
Hence it suffices to show that
$f_i$ and $f_j$ have no common zeros in $\Z/p\Z$ if $p$ is sufficiently large.
Since the polynomials $f_i(x)$ are irreducible in $\Z[x]$ they are irreducible in $\Q[x]$. 
If $i \neq j$ then $\gcd(f_i,f_j) = 1$ in $\Q[x]$, which is a Euclidean domain. Hence there exist 
polynomials $u_{ij}(x)$ and $v_{ij}(x)$ in $\Q[x]$ such that
$$u_{ij}(x) f_i(x) + v_{ij}(x) f_j(x) = 1.$$
Let $d_{ij}$ be the least common denominator of the coefficients of $u_{ij}(x)$ and $v_{ij}(x)$,
then $g_{ij}(x) = d_{ij} u_{ij}(x)$ and $h_{ij}(x) = d_{ij} v_{ij}(x)$ are in $\Z[x]$, and we have:
$$g_{ij}(x) f_i(x) + h_{ij}(x) f_j(x) = d_{ij}.$$
Suppose that $f_i(x)\ \md p$ and $f_j(x)\ \md p$ have a common root $r \in \Z/p\Z$ for some prime $p$.
Substituting $r$ for $x$ in the equation above and reducing modulo $p$ yields
$$d_{ij} \equiv 0\ (\md p),$$
meaning that $p$ divides $d_{ij}$, which is only possible for finitely many primes $p$, e.g. $p$ has to be smaller than $d_{ij}$. 
Hence for all sufficiently large primes $p$ the polynomials $f_i$ and $f_j$ have no common zeros in $\Z/p\Z$. This completes proof.
\end{proof}

The product that defines the Bateman--Horn constant need not converge absolutely.
Consequently, we must take care to justify its convergence.  
We are now ready to prove the main result of this section.

\begin{theorem}\label{Theorem:ProductConverge}
The product that defines $C(f_1,f_2,\ldots,f_k)$ converges.
\end{theorem}

\begin{proof}
Lemma \ref{Lemma:OmegaSplit} implies that
\begin{equation*}
\sum_{p\leq x} \frac{\omega(p) - k}{p} = \sum_{i=1}^k \sum_{p \leq x} \frac{\omega_i(p) - 1}{p}  + D
\end{equation*}
for all $x \geq 0$; the constant $D$ arises because of the finitely many exceptions to \eqref{eq:OmegaSum}.
The preceding lemma and Lemma \ref{Lemma:Converge} ensure that
\begin{equation}\label{eq:WeDidThat}
\sum_{p} \frac{k - \omega(p)}{p} \quad \text{converges}.
\end{equation}
Then a binomial expansion yields 
\begin{equation*}
\left(1 - \frac{1}{p} \right)^{-k} \left(1 - \frac{\omega(p)}{p} \right) = 1 + \frac{k - \omega(p)}{p} 
+ \frac{B(p)}{p^2} ,
\end{equation*}
in which
\begin{equation*}
B(p) = \frac{k(k-1)}{2} - \omega(p) + O\left( \frac{1}{p}\right)
\end{equation*}
is uniformly bounded because $|\omega(p)| \leq \deg f$.
Let
\begin{equation*}
a_p = \frac{k - \omega(p)}{p} + \frac{B(p)}{p^2}
\end{equation*}
and observe that 
\begin{equation*}
\sum_p a_p = \sum_p\frac{k - \omega(p)}{p} + \sum_p \frac{B(p)}{p^2}
\end{equation*}
converges by \eqref{eq:WeDidThat} and the comparison test.  Since
\begin{equation*}
|a_p|^2 
= \left(\frac{k - \omega(p)}{p} + \frac{B(p)}{p^2} \right)^2
= \frac{(k - \omega(p))^2}{p^2} + \frac{2B(p)(k - \omega(p))}{p^3} + \frac{B(p)^2}{p^4},
\end{equation*}
the comparison test ensures that $\sum_p |a_p|^2$ converges.
Consequently, Lemma \ref{Lemma:Trick} tells us that
$\prod_p (1+a_p)$, the product that defines $C(f_1,f_2,\ldots,f_k)$, converges.
\end{proof}

The preceding argument, first envisioned in its general form by Bateman and Horn (but also in some special cases by Nagell (1921), Rademacher (1924) and Ricci (1937); see~\cite{DavenportSchinzel} for a discussion), shows that the constant $C(f_1,f_2,\ldots,f_k)$ is well defined.
However it is still hard to compute due to the fact that the convergence of the product in question is not necessarily absolute or rapid. 
This consideration leaves an open problem: express the constant $C(f_1,f_2,\ldots,f_k)$ in terms of an absolutely convergent product. This was done in some special cases in a subsequent paper~\cite{BatemanHorn-2} of Bateman and Horn, and then generally by Davenport and Schinzel~\cite{DavenportSchinzel}. 
Several methods to accelerate the convergence rate of infinite products 
for approximation purposes use $L$-functions; see~\cite{moree, Jacobson}.

\section{Single polynomials}\label{Section:Single}

The Bateman--Horn conjecture implies a wide range of known theorems and 
unproved conjectures.  In this section we examine several such results in the case of a single polynomial. 
This provides us with some practical experience computing Bateman--Horn constants
and it also highlights some delicate convergence issues.
Applications of the Bateman--Horn conjecture to families of two or more polynomials
are studied in Section \ref{Section:Multiple}. 

\subsection{Prime number theorem for arithmetic progressions}
In 1837, Peter Gustav Lejeune Dirichlet (1805--1859) proved 
that if $a$ and $b$ are relatively prime natural numbers, 
then there are infinitely many primes of the form $at+b$, in which $t\in \N$.
For example, there are infinitely many primes that end in $123{,}456{,}789$.
To see this, apply Dirichlet's result with $a = 10{,}000{,}000$ and
$b = 123{,}456{,}789$.\footnote{The values of $t \leq 100$ for which
$at+b$ is prime are $11, 29, 43, 50, 59, 64, 68, 73, 97, 98$.}

Let $\pi_{a,b}(x)$ denote the number of primes at most $x$ that are of the form $at+b$.
The complex-variables proof of the prime number theorem can be modified to provide the following
asymptotic formulation of Dirichlet's result \cite{TaoDirichlet} (see \cite{Selberg} and
the discussion on \cite[p.~236]{Pollack} for information about elementary approaches).

\begin{theorem}[Prime Number Theorem for Arithmetic Progressions]\label{Theorem:PNTAP}
If $a$ and $b$ are relatively prime natural numbers, then
\begin{equation}\label{eq:PNTAP}
\pi_{a,b}(x) \ \sim \ \frac{1}{\phi(a)} \Li(x).
\end{equation}
\end{theorem}

Here 
\begin{equation*}
\phi(n) \ = \ 
\#\big\{ k \in \{1,2,\ldots,n\} : \gcd(k,n) = 1\big\}
\end{equation*}
denotes the \emph{Euler totient function}.  Its value equals the order of the group
$(\Z/n\Z)^{\times}$ of units in $\Z/n\Z$.  The totient function is multiplicative, in the sense that
$\phi(mn) = \phi(m) \phi(n)$ whenever $\gcd(m,n) = 1$.  It enjoys the product decomposition
\begin{equation}\label{eq:PhiProduct}
\phi(n) = n\prod_{p|n} \left(1 - \frac{1}{p} \right),
\end{equation}
in which the expression $p|n$ denotes that the product is taken over all primes $p$ that divide $n$.
For example, $\phi(6) = 2$ since only $1$ and $5$ are in the range $\{1,2,\ldots,6\}$ and relatively prime to $6$.
The product formulation \eqref{eq:PhiProduct} tells us the same thing:
\begin{equation*}
\phi(6) = 6 (1 - 1/2)(1 - 1/3) = 6(\tfrac{1}{2})(\tfrac{2}{3}) = 2.
\end{equation*}

What is the intuitive explanation behind the prime number theorem for arithmetic progressions?
If $\gcd(a,b) \neq 1$, then $a$ and $b$ share a common factor and hence
$at+b$ is prime for at most one $t$.  Thus, $\gcd(a,b) = 1$ is a necessary condition for the polynomial $at+b$
to generate infinitely many primes.  For each fixed $a$, 
this yields exactly $\phi(a)$ admissible values of $b \pmod{a}$.
Since the prime number theorem tells us that $\pi(x) \sim \Li(x)$,
\eqref{eq:PNTAP} tells us that each of the $\phi(a)$ admissible congruence classes modulo $a$
receives an approximately equal share of primes.

The prime number theorem for arithmetic progressions
(Theorem \ref{Theorem:PNTAP}) is a straightforward consequence of the Bateman--Horn conjecture.
Let $f(t) = at+b$, in which $\gcd(a,b) = 1$.  Then 
\begin{equation}\label{eq:DAPCong}
f(t) \equiv 0 \pmod{p}
\quad\iff\quad
at \equiv - b \pmod{p}.
\end{equation}
If $p \nmid a$, then $a$ is invertible modulo $p$ and the preceding congruence has a unique solution.
If $p|a$, then \eqref{eq:DAPCong} has no solutions since $\gcd(a,b)=1$.  Therefore,
\begin{equation*}
\omega_f(p) =
\begin{cases}
1 & \text{if $p \nmid a$},\\
0 & \text{if $p| a$},
\end{cases}
\end{equation*}
and hence
\begin{equation*}
C(f;p) 
= \prod_p \left( 1 - \frac{1}{p} \right)^{-1}
\left(1- \frac{\omega_f(p)}{p} \right) 
= \prod_{p|a} \left( 1 - \frac{1}{p} \right)^{-1}
= \frac{a}{\phi(a)}
\end{equation*}
by \eqref{eq:PhiProduct}.  In particular, the potentially infinite product reduces to a finite product
indexed only over the prime divisors of $a$.  
Since
\begin{equation*}
at+b \leq x 
\quad\iff\quad
t \leq \frac{x-b}{a},
\end{equation*}
we have
\begin{align*}
\pi_{a,b}(x)
&= Q\bigg(f;\frac{x-b}{a}\bigg)\\
&\sim \frac{a}{\phi(a)} \cdot \frac{(x-b)/a}{\log((x-b)/a)} \\
&= \frac{a}{\phi(a)} \cdot \frac{(x/a) - (b/a)}{\log(x-b) - \log a} \\
&\sim \frac{a}{\phi(a)} \cdot  \frac{x/a}{\log(x-b)} \\
&\sim \frac{x}{\phi(a)\log x} \sim \frac{1}{\phi(a)} \Li(x),
\end{align*}
which is the desired result. 

The weaker statement about simply the infinitude of primes in an arithmetic progression is a special case of the Bunyakovsky conjecture and is currently the only case of this conjecture that has been settled. 
The conjecture is open for quadratic and cubic polynomials, as we discuss next.

\subsection{Landau's conjecture and its relatives}\label{Section:Landau}
In our heuristic argument (Section \ref{Section:Heuristic}),
we explained how Landau's conjecture (there are infinitely many primes of the form $n^2+1$)
follows from the Bateman--Horn conjecture.  For $f(t) = t^2+1$, we showed that
\begin{equation*}
Q(f;x) \ \sim \ (0.68640\ldots) \Li(x);
\end{equation*}
in particular, the conjecture suggests that Landau's intuition was correct.  
Let $\pi_{\text{Landau}}(x)$ denote the number of primes of the form $n^2+1$ that are at most $x$.  Since
\begin{equation*}
t^2 + 1 \leq x
\quad \iff \quad
t \leq \sqrt{x-1},
\end{equation*}
it follows that
\begin{align*}
\pi_{\text{Landau}}(x)
&= Q(f; \sqrt{x-1}) \sim (0.68640\ldots) \Li(\sqrt{x-1}) \\
&\sim (0.68640\ldots) \frac{\sqrt{x-1}}{\log(\sqrt{x-1})} \\
& \sim ( 1.3728\ldots) \frac{\sqrt{x}}{\log x}.
\end{align*}
Thus, $\pi_{\text{Landau}}(x)$ grows like a constant times $\pi(x) / \sqrt{x}$.  

The Bateman--Horn conjecture also implies important variants of Landau's conjecture.
For example, Friedlander and Iwaniec proved that there are infinitely many primes of the form
$x^2 + y^4$ (they also provided asymptotics for the counting function of such primes) \cite{FI2}.
For each fixed $y\geq 1$, the Bateman--Horn conjecture suggests that there are infinitely many primes of the
form $x^2 + y^4$.  
A result of Heath-Brown~\cite{heath-brown} guarantees the existence of infinitely many primes
(with an asymptotic formula for the growth of their number) of the form $x^3+2y^3$, thereby confirming the conjecture of Hardy and Littlewood on the infinitude of primes expressible as a sum of three cubes. These are results in the interesting and promising direction of representing primes
by multivariate polynomials, see the survey \cite{moroz} 
and the recent preprint \cite{Destagnol}.

Let us briefly turn to cubic polynomials in one variable. A result of~\cite{foo_zhao} states, roughly speaking, that on the average polynomials of the form $t^3+k$ for \highlight{squarefree $k > 1$} assume infinitely many prime values at integer points, in some well-defined sense.
We are not aware of a definitive published result on any specific example of such a polynomial: an existence of infinitude of prime values of a cubic polynomial is a special case of the Bunyakovsky conjecture that is sometimes called the ``cubic primes conjecture.'' 
For example, $f(t) = t^3-2$ is irreducible and does not vanish identically modulo any prime.
The Bateman--Horn conjecture predicts that this polynomial assumes prime values infinitely often.

\subsection{Tricking Bateman--Horn?}
What happens if we replace $n^2+1$ with $n^2-1 = (n-1)(n+1)$?
The only prime of this form is $3$.
Of course, the polynomial in question is reducible and hence is
not even a permissible candidate for the conjecture.  
Does the Bateman--Horn conjecture ``detect this'' attempted fraud, 
or does it just plow ahead and suggest to the unwary that there are infinitely many primes of this form?  
For the sake of curiosity, let us try it and see what happens.

If $f(n) = n^2-1$, then $f(n) \equiv 0 \pmod{p}$ becomes $n^2 \equiv 1 \pmod{p}$ and hence
\begin{equation*}
\omega_f(p) = 
\begin{cases}
1 & \text{if $p=2$},\\
2 & \text{otherwise}.
\end{cases}
\end{equation*}
Thus,
\begin{equation}\label{eq:Cfp1}
C(f) = \prod_{p \geq 3} \frac{p-2}{p-1} = \prod_{p \geq 3} \left(1 - \frac{1}{p-1}\right).
\end{equation}
Let $P_n$ denote the set of the first $n$ odd primes.
For example, $P_1 = \{3\}$, $P_2 = \{3,5\}$, $P_3 = \{3,5,7\}$, and so forth.
Numerical evidence (Table \ref{Table:DivergeZero}) suggests that 
\begin{equation}\label{eq:PrimeEuler}
\lim_{n\to\infty} \prod_{p \in P_n} \left(1 - \frac{1}{p-1}\right) \ = \ 0;
\end{equation}
that is, the product that defines $C(f)$ diverges to zero (this is the case).
If this application of the Bateman--Horn conjecture were admissible (it is not since $f$ is reducible), 
we would expect no primes of the form $n^2-1$.
This is not too far from the truth:  we were off by only one.
The Bateman--Horn conjecture is surprisingly robust; in some sense, it detected
our trickery and rejected it.

\begin{table}
\begin{equation*}
\begin{array}{c|c}
n & \prod_{p \in P_n} \big(1 - \frac{1}{p-1}\big) \\
\hline
10 & 0.210114 \\
100 & 0.117208 \\
1{,}000 & 0.0824772 \\
10{,}000 & 0.0641136 \\
100{,}000 & 0.0526554 \\
1{,}000{,}000 & 0.044777\\
10{,}000{,}000 & 0.0390052\\
\end{array}
\end{equation*}
\caption{The partial products $\prod_{p \in P_n} \left(1 - \frac{1}{p-1}\right)$ appear to
diverge to zero.}
\label{Table:DivergeZero}
\end{table}

Why does \eqref{eq:Cfp1} diverge to zero? 
Euler's result (Theorem \ref{Theorem:Euler}) ensures that
\begin{equation*}
\sum_{p \in P_n} \frac{1}{p-1} > \sum_{p \in P_n} \frac{1}{p},
\end{equation*}
which diverges as $n\to\infty$.  An application of Lemma \ref{Lemma:Trick}
implies that \eqref{eq:Cfp1} diverges (to zero); that is, $C(f) = 0$.
Thus, the Bateman--Horn conjecture detected, in a subtle way, the difference between
the polynomials $n^2+1$ (which is believed to generate infinitely many prime values)
and $n^2-1$, which is prime exactly once.

\subsection{Prime-generating polynomials}
Euler observed in 1772 that the polynomial $f(t) = t^2+t+41$ assumes prime values for $t=0,1,\ldots,39$.
However, $f(40) = 1681 = 41^2$
is composite.
Is there a nonconstant polynomial that assumes only prime values?

\begin{theorem}
Let $f \in \Z[x]$.  If $f(n)$ is prime for all $n \geq 0$, then $f$ is constant.
\end{theorem}

\begin{proof}
Let $p = f(0)$, which is prime by assumption.  
For each $n \geq 0$, the prime $f(pn)$ is divisible by $p$.
Then $f(pn) = p$ for $n\geq 0$ and hence $f(pn) - p$
has infinitely many roots and is therefore zero.  Thus, $f$ is the constant polynomial $p$. 
\end{proof}

This shows that no single-variable polynomial can assume only prime values 
for all natural arguments.  Surprisingly, there is a polynomial of degree twenty-five in twenty-six variables whose positive integral range is precisely the set of prime numbers \cite{Jones}.  This startling fact is related to Matiyasevich's solution to Hilbert's tenth problem
\cite{Matiyasevich} and the work of Davis--Putnam--Robinson \cite{Davis}. It is not known what is the smallest number of variables a prime-generating polynomial must have, but it is definitely less than twenty-six: a polynomial with this property in twelve variables is also known; see \cite[Sect.~2.1]{luca}.

What is so special about $41$?
Suppose that $f(t) = t^2+t+k$ generates primes for the first
few nonnegative integral values of $t$.  Then $k = f(0)$ is prime.
In 1913, Georg Yuri Rainich (1886--1968) proved if $p$ is prime, then
$n^2 + n + p$ is prime for $n=0,1,\ldots,p-2$ if and only if the imaginary quadratic field 
$\Q(\sqrt{1 - 4p})$ has class number one \cite{Rabinowitsch}\footnote{Rainich published 
\cite{Rabinowitsch} under his original birth name, Rabinowitsch.  According to \cite{Palka},
``Rainich was giving a lecture in which he made use of a clever trick which he had discovered. Someone in the audience indignantly interrupted him pointing out that this was the famous Rabinowitsch trick and berating Rainich for claiming to have discovered it. Without a word Rainich turned to the blackboard, picked up the chalk, and wrote `RABINOWITSCH.'
He then put down the chalk, picked up an eraser and began erasing letters. When he was done what remained was `RA IN  I CH.'
He then went on with his lecture.''}; for our purposes it suffices to say that this means that $\Q(\sqrt{1 - 4p})$ is a unique
factorization domain.
The Baker--Heegner--Stark theorem ensures that there are only finitely
many primes $p$ with this property \cite{Baker, Heegner, Stark, StarkGap}.
The largest of these, $p = 41$, corresponds to the quadratic field $\Q(\sqrt{-163})$.
Thus, we cannot beat Euler at his own game.

Perhaps we can beat Euler on average.  Can we find an Euler-type polynomial
that produces an asymptotically greater number of primes than Euler's polynomial?
Let us first examine what the Bateman--Horn conjecture says about $f(t) = t^2 + t + 41$.

Since $f(t)$ is identically $1$ modulo $2$, $\omega_f(2) = 0$.
In what follows we use the ``completing the square'' identity
\begin{equation}\label{eq:CompleteSquare}
4a(at^2+ bt + c) = (2at+b)^2 -(b^2 - 4ac).
\end{equation}
For $p \geq 3$, this ensures that
\begin{equation*}
t^2+t+41 \equiv 0 \pmod{p}
\quad\iff\quad
(2t+1)^2 \equiv -163 \pmod{p}.
\end{equation*}
Thus, everything boils down to whether $-163$ is a quadratic residue or nonresidue modulo the odd prime $p$:
\begin{equation*}
\omega_f(p)
= 1 + \left(\frac{-163}{p} \right).
\end{equation*}
Here $(\frac{-163}{p})$ is a \emph{Legendre symbol}, defined by
\begin{equation*}
\left(\frac{\ell}{p}\right)
= 
\begin{cases}
0 & \text{if $p|\ell$},\\
1 & \text{if $\ell$ is a quadratic residue modulo $p$},\\
-1 & \text{if $\ell$ is a quadratic nonresidue modulo $p$}.
\end{cases}
\end{equation*}
Numerical computation confirms that $-163$ is a quadratic nonresidue modulo 
\begin{equation}\label{eq:Prime37}
3, 5, 7, 11, 13, 17, 19, 23, 29, 31, 37,
\end{equation}
the first eleven odd primes.  Thus, $\omega_f(p) = 0$ for these primes and hence
\begin{align}
C(t^2+t+ 41)
&= \prod_p \left(1 - \frac{1}{p} \right)^{-1} \left(1 - \frac{\omega_f(p)}{p}\right) \nonumber \\
&= 2 \prod_{3 \leq p \leq 37} \frac{p}{p-1} \prod_{p \geq 41} \left(     1 + \frac{1-\omega_f(p)}{p-1} \right) \label{eq:OhNo}\\
&\approx 2 \cdot 3.31993 = 6.63985. \label{eq:EulerConstant}
\end{align}
The factors corresponding to $p=2,3,\ldots,37$ are each greater than one,
which drives $C(f)$ up.  We have little control over the second product, although we may
hope that $1-\omega_f(p)$ changes sign regularly enough to keep it in check.
Although it is not clear at first glance that the second product in \eqref{eq:OhNo}
converges, the product that defines the Bateman--Horn constant
is guaranteed to converge (see Section \ref{Section:Converge}) and thus the second product must as well.

The Bateman--Horn conjecture suggests that
\begin{equation}\label{eq:EulerLi}
Q(t^2+t+41;x)
\sim (3.31993\ldots) \Li(x).
\end{equation}
Can we find a second-degree polynomial $f(t)$ for which $Q(f;x)$ exceeds this amount
asymptotically?  
To this end, we want each factor in the product \eqref{eq:C} to be as large as possible.
Unfortunately, we cannot arrange for $\omega_f(p) = 0$ for all primes $p$ since the corresponding infinite product 
\begin{equation*}
\prod_p \left(1 - \frac{1}{p}\right)^{-1} = \prod_p \frac{p}{p-1} = \prod_p \left(1 + \frac{1}{p-1}\right)
\end{equation*}
would diverge by Lemma \ref{Lemma:Trick} and Theorem \ref{Theorem:Euler}.  However, this would contradict
Theorem \ref{Theorem:ProductConverge}.

In fairness to Euler, we should try to beat him with a polynomial of the same type.
Thus, we search for an integer $k$ such that the polynomial $f(t) = t^2+t+k$
satisfies $\omega_f(p) = 0$ for the first several dozen or so primes.
We first need $k \equiv 1 \pmod{2}$ such that $\omega_f(2) = 0$.
The identity \eqref{eq:CompleteSquare} shows that for odd $p$,
\begin{equation*}
f(t) \equiv 0 \pmod{p}
\quad\iff\quad
(2t+1)^2 \equiv 1-4k \pmod{p}.
\end{equation*}
Consequently, we need to choose an odd $k$ such that
$1-4k$ is a quadratic nonresidue modulo $p$ for a long initial string of odd primes.

Let $P_n$ denote the set of odd primes at most $n$.
For each $p \in P_n$, let $r_p$ be a quadratic nonresidue modulo $p$.
The Chinese Remainder Theorem provides an odd $k_n$, unique modulo $2\prod_{p\in P_n}p$,
such that $k_n \equiv 4^{-1}(1-r_p) \pmod{p}$ for each $p \in P_n$.
Then $1-4k_n \equiv r_p \pmod{p}$ is a quadratic nonresidue and hence $\omega_p(f) = 0$
for each $p \in P_n$.
The corresponding Bateman--Horn constant is
\begin{equation*}
C(t^2+t+k_n)
= 2 \prod_{3 \leq p \leq n} \frac{p}{p-1} \prod_{p > n} \frac{p - \omega_f(p)}{p-1}.
\end{equation*}
If $n = 547$, the hundredth odd prime, and we let $r_p$ equal the least primitive root of $p$, the corresponding constant
\begin{equation*}
C(t^2+t+k_{100}) 
\approx 2 \cdot (5.4972\ldots) = 10.9945
\end{equation*}
easily beats the constant \eqref{eq:EulerConstant} corresponding to Euler's polynomial.  Unfortunately, $k_{100}$
is not as easily remembered as Euler's $41$:
\begin{quote}
3682528442873462645493394982418837604455310384084190749577\\
5453041420103519734083583186615204669729662489042369819157\\
7358565650719425670030967384568941667322171286195075149379\\
680113340447535104953498545635385597443028681.
\end{quote}
It is conceivable that other choices of $r_p$ might lead to a smaller constant, although we have not looked into the matter.\footnote{If we choose the least primitive roots
$2, 2, 3, 2, 2, 3, 2, 5, 2, 3, 2$ of the primes \eqref{eq:Prime37}, respectively, and apply the algorithm
above we obtain
$k_{37} = 1{,}448{,}243{,}016{,}041$.}
The Bateman--Horn conjecture suggests that
\begin{equation*}
Q(t^2+t+k_{100};x) \sim (5.4972\ldots) \Li(x),
\end{equation*}
which is asymptotically larger than the corresponding prediction \eqref{eq:EulerLi} for Euler's polynomial.

Before we pat ourselves on the back for beating Euler, we should point out that 
the search for prime-producing polynomials using these sorts of arguments
has a long history \cite{Jacobson, Beeger, Fung}.  Moreover, without the Bateman--Horn conjecture
or one of its weaker relatives (Section \ref{Section:Prehistory}), we do not even know if any
quadratic polynomial produces infinitely many primes.  Thus, this is all speculative.

\subsection{A conjecture of Hardy and Littlewood}
A general conjecture about the asymptotic distribution of prime values assumed by quadratic polynomials is due to 
G.H.~Hardy (1877--1947) and John E. Littlewood (1885--1977) \cite[p.~48]{Hardy} (see also \cite[p.~19]{HardyWright}).
The more convenient formulation below is from \cite[p.~499]{Jacobson}.

\medskip\noindent\textbf{Hardy--Littlewood Conjecture (F).}
\emph{If $a, b, c$ are relatively prime integers, $a$ is positive, $a+b$ and $c$ are not both even, 
and $b^2-4ac$ is not a perfect square, then there are infinitely many primes of the form 
$f(t) = at^2+bt+c$.
The number of such primes at most $x$ is asymptotic to
\begin{equation}\label{eq:HLF}
\epsilon \prod_{\substack{p \geq 3 \\ p|\gcd(a,b)}} \frac{p}{p-1} \prod_{\substack{p \geq 3 \\ p \nmid a}} \left(1 - \frac{(\Delta/p)}{p-1} \right) \Li(x),
\end{equation}
in which
\begin{equation*}
\epsilon = 
\begin{cases}
\frac{1}{2} & \text{if $2\nmid(a+b)$},\\
1 & \text{otherwise}.
\end{cases}
\end{equation*}
}
\medskip

This is a consequence of the Bateman--Horn conjecture.  Let us see why, paying careful attention to the relevance
of Hardy and Littlewood's hypotheses.  
Suppose that $f(t) = at^2+bt+c$, in which $a>0$. 
What conditions on $a,b,c$ are necessary for $f$ to be prime infinitely often?
Since
\begin{equation*}
at^2+bt+c
\equiv 
\begin{cases}
c & \text{if $t \equiv 0 \pmod{2}$},\\
a+b+c & \text{if $t \equiv 1 \pmod{2}$},
\end{cases}
\end{equation*}
we want either $a+b$ or $c$ (or both) to be odd.
Consequently, 
\begin{equation}\label{eq:HardyOmega2}
\omega_f(2) = 
\begin{cases}
0 & \text{if $a+b$ is even and $c$ is odd},\\
1 & \text{if $a+b$ is odd and $c$ is odd},\\
1 & \text{if $a+b$ is odd and $c$ is even}.
\end{cases}
\end{equation}
Suppose that $p$ is an odd prime.  There are two cases.
\begin{itemize}
\item If $p|a$, then $f(t) \equiv bt + c \pmod{p}$.  Since $\gcd(a,b,c) = 1$, we conclude that
\begin{equation*}
\omega_f(p) = 
\begin{cases}
0 & \text{if $p|b$},\\
1 & \text{if $p \nmid b$}.
\end{cases}
\end{equation*}

\item If $p\nmid a$, then \eqref{eq:CompleteSquare} ensures that
\begin{equation*}
f(t) \equiv 0 \pmod{p}
\qquad \iff\ \qquad
(2at+b)^2 \equiv \Delta \pmod{p},
\end{equation*}
in which $\Delta= b^2-4ac$ is the discriminant of $f$.  Thus,
\begin{equation*}
\omega_f(p) = 1+\left(\frac{\Delta}{p} \right).
\end{equation*}
\end{itemize}
Thus, the Bateman--Horn constant \eqref{eq:C} is 
\begin{align}
C(f)
&= \big(2-\omega_f(2) \big) \prod_{\substack{p\geq 3,\,p|a\\p|b}}  \frac{p-0}{p-1}
 \prod_{\substack{p\geq 3,\,p|a\\p\nmid b}}   \frac{p-1}{p-1} 
 \prod_{\substack{p\geq 3\\p \nmid a}} \frac{p-(1+(\Delta/p))}{p-1} \nonumber \\
&= 2 \epsilon  \prod_{\substack{p\geq 3\\p|\gcd(a,b)}}  \frac{p}{p-1}
\prod_{\substack{p\geq 3\\p \nmid a}}  \left(1 - \frac{(\Delta/p)}{p-1} \right). \label{eq:HLC}
\end{align}

There is a subtle point here that we wish to highlight.
If $\Delta$ is a perfect square, then $(\Delta/p) = 1$ and the second factor in \eqref{eq:HLC} diverges (to zero) by
Lemma \ref{Lemma:Trick} and Theorem \ref{Theorem:Euler}.  
This does not contradict the Bateman--Horn conjecture, since $f$ is not irreducible in this case.
If $\Delta$ is a perfect square, then the two roots
\begin{equation*}
\frac{-b + \sqrt{\Delta}}{2a}
\qquad\text{and}\qquad
\frac{-b - \sqrt{\Delta}}{2a}
\end{equation*}
of $f$ belong to $\Q$.  Then $f$ would be reducible over $\Q$ and hence, by Gauss' lemma \cite[Prop.~5, p.~303]{DF}, reducible over $\Z$.
Thus, $\Delta$ cannot be a perfect square if $f$ is to be prime infinitely often:  this is why
Hardy and Littlewood assume that $b^2-4ac$ is not a perfect square.
If $\Delta$ is not a perfect square, then the prediction \eqref{eq:Q} of the Bateman--Horn conjecture
provides the asymptotic formula \eqref{eq:HLF} proposed by Hardy and Littlewood.

\subsection{Ulam's spiral}
\begin{figure}
    \centering
        \includegraphics[width=0.5\textwidth]{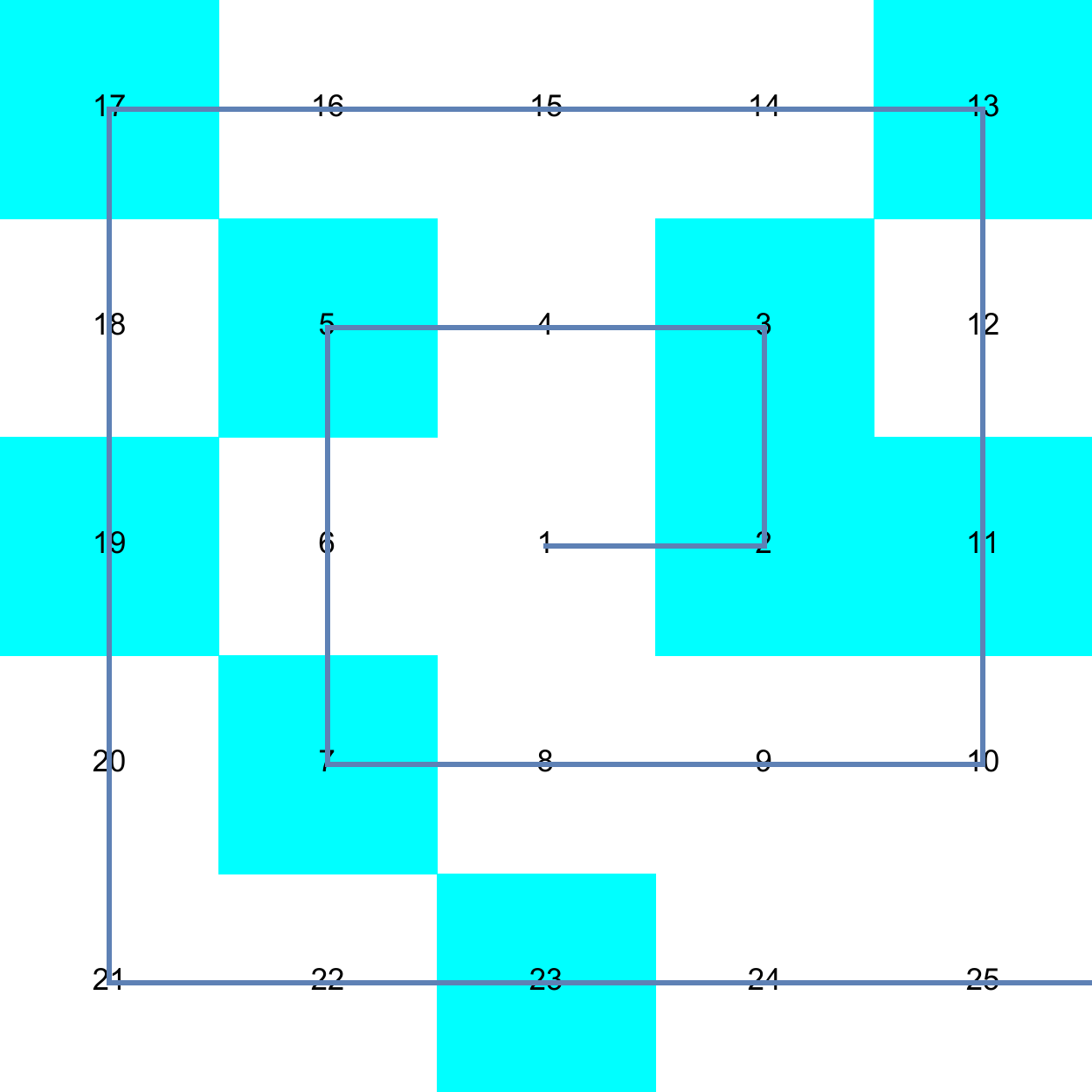}
    \caption{The natural numbers spiral outward counterclockwise from the origin.  A colored box
    is placed over each prime.}
    \label{Figure:Ulam1}
\end{figure}

In 1963, Stanis\l aw Ulam (1909--1984) discovered a startling pattern in the primes, allegedly while doodling at a scientific meeting; see Figure \ref{Figure:Ulam1}.
The story was popularized by Martin Gardner (1914--2010) in his much-loved \emph{Scientific American} column ``Mathematical Games'' \cite{Gardner}:
\begin{quote}\small
Last fall Stanislaw M. Ulam of the Los Alamos Scientific Laboratory, attended a scientific meeting at which he found himself listening to what he describes as a ``long and very boring paper.'' To pass the time he doodled a grid of horizontal and vertical lines on a sheet of paper. 
His first impulse was to compose some chess problems, then he changed his mind and began to number the intersections, starting near the center with $1$ 
and moving out in a counterclockwise spiral. With no special end in view, he began circling all the prime numbers. 
To his surprise the primes seemed to have an uncanny tendency to crowd into straight lines.
\end{quote}

The patterns observed by Ulam are evident in Figure \ref{Figure:Ulam2}.  There are certain diagonals that the primes
prefer and others that they eschew.  Less prominent, but still noticeable, are the scarcity or abundance of primes
on some horizontal or vertical lines.  Others seem to have more than their fair share of primes.  
The primes, which are often assumed to be ``random'' in their overall distribution (Section \ref{Section:PNT}), 
manage to conspire over great distances to form these intriguing patterns.  What is the explanation for this
behavior?

\begin{figure}
    \centering
    \begin{subfigure}[t]{0.475\textwidth}
        \centering
        \includegraphics[width=\textwidth]{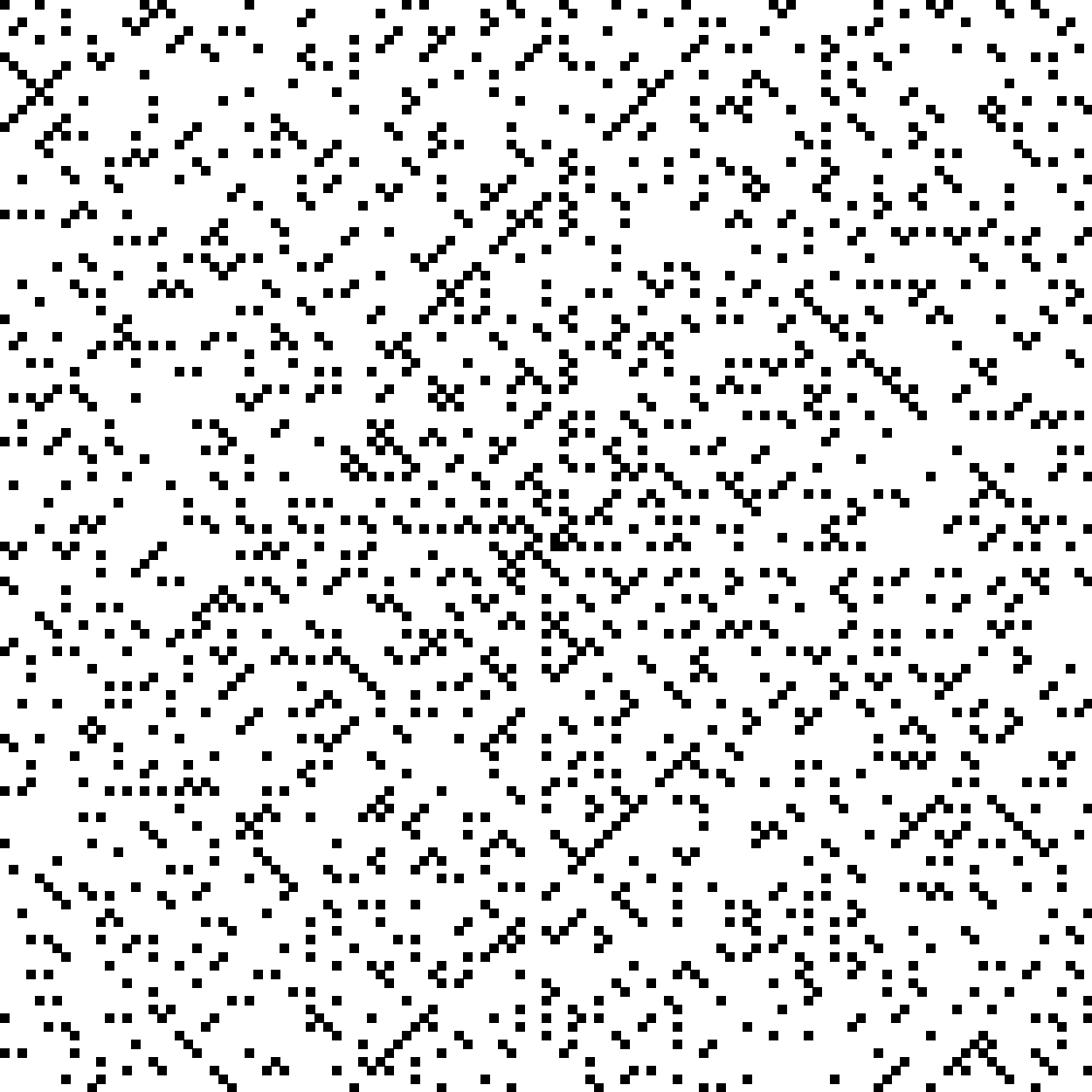}
        \caption{$125 \times 125$}
    \end{subfigure}
    \quad
    \begin{subfigure}[t]{0.475\textwidth}
        \centering
        \includegraphics[width=\textwidth]{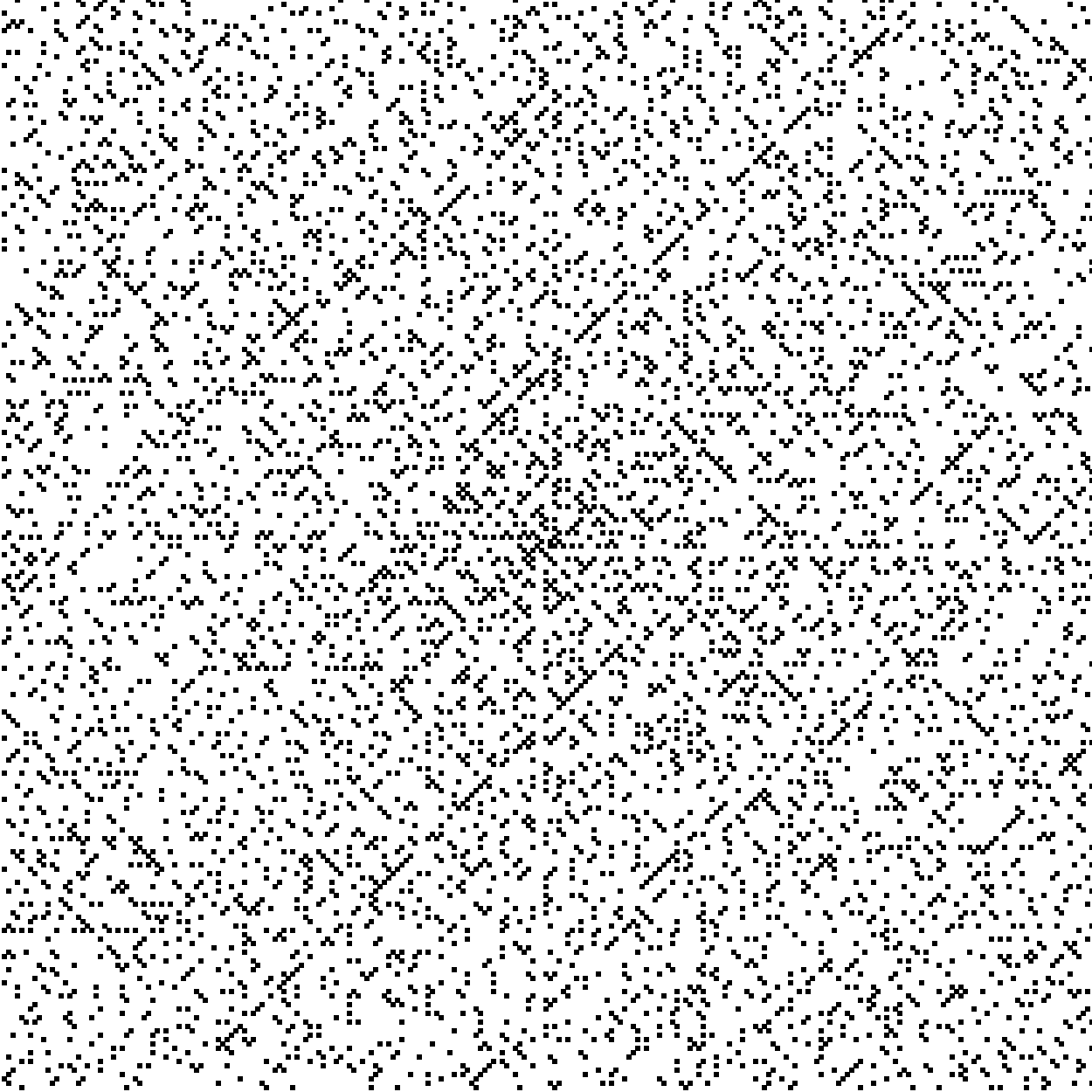}
        \caption{$250 \times 250$}
    \end{subfigure}
    \caption{Plots of the Ulam spiral on grids of several sizes.   
There are certain diagonals that the primes (black)
prefer and others that they eschew.  Less prominent, but still noticeable, are the scarcity or abundance of primes
on some horizontal or vertical lines.  The existence of these patterns is a consequence of the Bateman--Horn conjecture.}
    \label{Figure:Ulam2}
\end{figure}

\begin{figure}
    \centering
        \includegraphics[width=\textwidth]{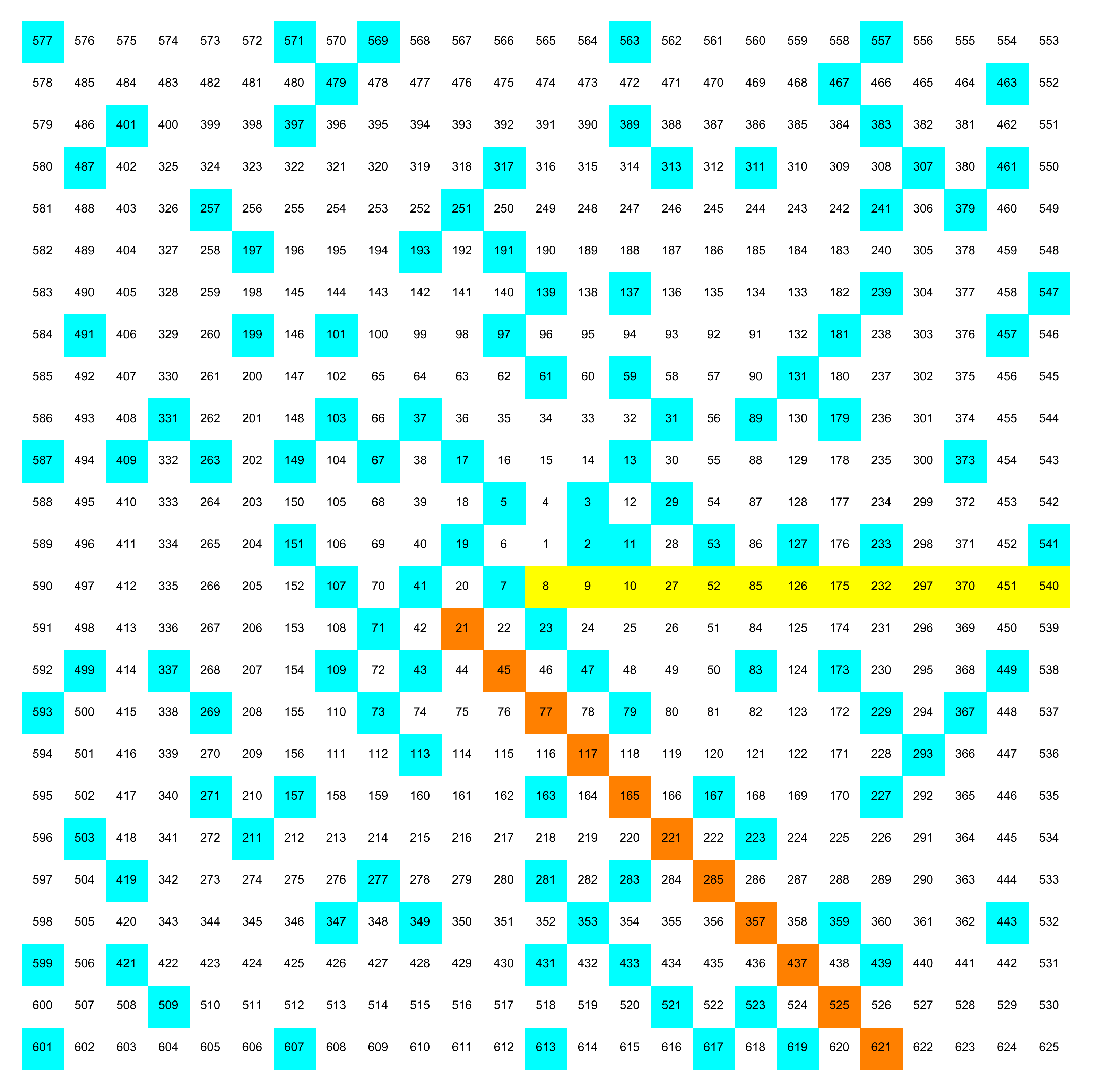}
    \caption{The horizontal ray depicted in yellow is prime free.
    If we ignore the initial $8$ and $9$, and start with $f(1) = 10$, then the $n$th element on this list is $f(n) = (4n+1)(n+1)$, which is 
    composite (see Example \ref{Example:HorizontalPrimeFree}).  
    Similarly, the diagonal ray depicted in orange is prime free.
    The $n$th number on this ray is $4n^2 + 12n + 5 = (2n+1)(2n+5)$   
(see Example \ref{Example:DiagonalPrimeFree}).}
    \label{Figure:Horizontal}
\end{figure}

In what follows, it is more fruitful to consider ``rays'' in the Ulam spiral instead of ``lines.''
This is no loss of generality since each line is the union of two rays.  

\begin{example}\label{Example:HorizontalPrimeFree}
Consider Figure \ref{Figure:Horizontal}, in which the horizontal ray 
\begin{equation}\label{eq:HorrayBad}
8,9,10,27,52,85,126,175,232,297,370,451,540,\ldots
\end{equation}
in the Ulam spiral appears devoid of primes.  Why does this occur?
Let us truncate our sequence slightly to avoid the short stretch of consecutive
integers at the beginning.  This yields the sequence
\begin{equation}\label{eq:UlamHor}
10,27,52,85,126,175,232,297,370,451,540,\ldots.
\end{equation}
To pass from $10$ to $27$, we walk around the exterior of the $3 \times 3$ square
\begin{equation*}
\begin{array}{ccc}
5 & 4 & 3 \\
6 & 1 & 2 \\
7 & 8 & 9
\end{array}
\end{equation*}
and take one more step; this requires $4 \cdot 4 + 1 = 17$ total steps.  Similarly, to pass from $27$ to $52$
we must traverse the exterior of a $5 \times 5$ square and take an additional step; this requires
$4 \times 6 + 1 = 25$ total steps.  Let $f(n)$ denote the $n$th number on the list \eqref{eq:UlamHor}.  Then
induction confirms that
\begin{equation*}
f(n) - f(n-1)  \ = \ 8n+1,
\end{equation*}
and hence
\begin{align}
f(n) 
&= \sum_{i=2}^n \big( f(i) - f(i-1) \big) + f(1)\nonumber \\
&= 10 + \sum_{i=2}^n (8i+1)\nonumber \\
&= 10 + (n-1) + 8\sum_{i=2}^n i \nonumber\\
&= n+9 + 8 \left( \frac{n(n+1)}{2} - 1\right) \nonumber\\
&= 4n^2 + 5n + 1 \nonumber \\
&= (4n+1)(n+1).  \label{eq:UlamBad}
\end{align}
This ensures that none of the numbers on the horizontal ray \eqref{eq:HorrayBad} is prime.
\end{example}

\begin{example}\label{Example:DiagonalPrimeFree}
The diagonal ray $21,45,77,117,165,221,285,357,437,525,621,\ldots$ in Figure \ref{Figure:Horizontal}
is similarly devoid of primes.  An argument similar to that used in Example 
\ref{Example:HorizontalPrimeFree} confirms that the $n$th number on this list is
$4n^2 + 12n + 5 = (2n+1)(2n+5)$.
\end{example}

The prime-free rays of Examples \ref{Example:HorizontalPrimeFree} 
and \ref{Example:DiagonalPrimeFree} (see Figure \ref{Figure:Horizontal})
are governed by a reducible quadratic polynomial.  What about prime-rich rays?

\begin{figure}
    \centering
    \begin{subfigure}[t]{0.475\textwidth}
        \centering
        \includegraphics[width=\textwidth]{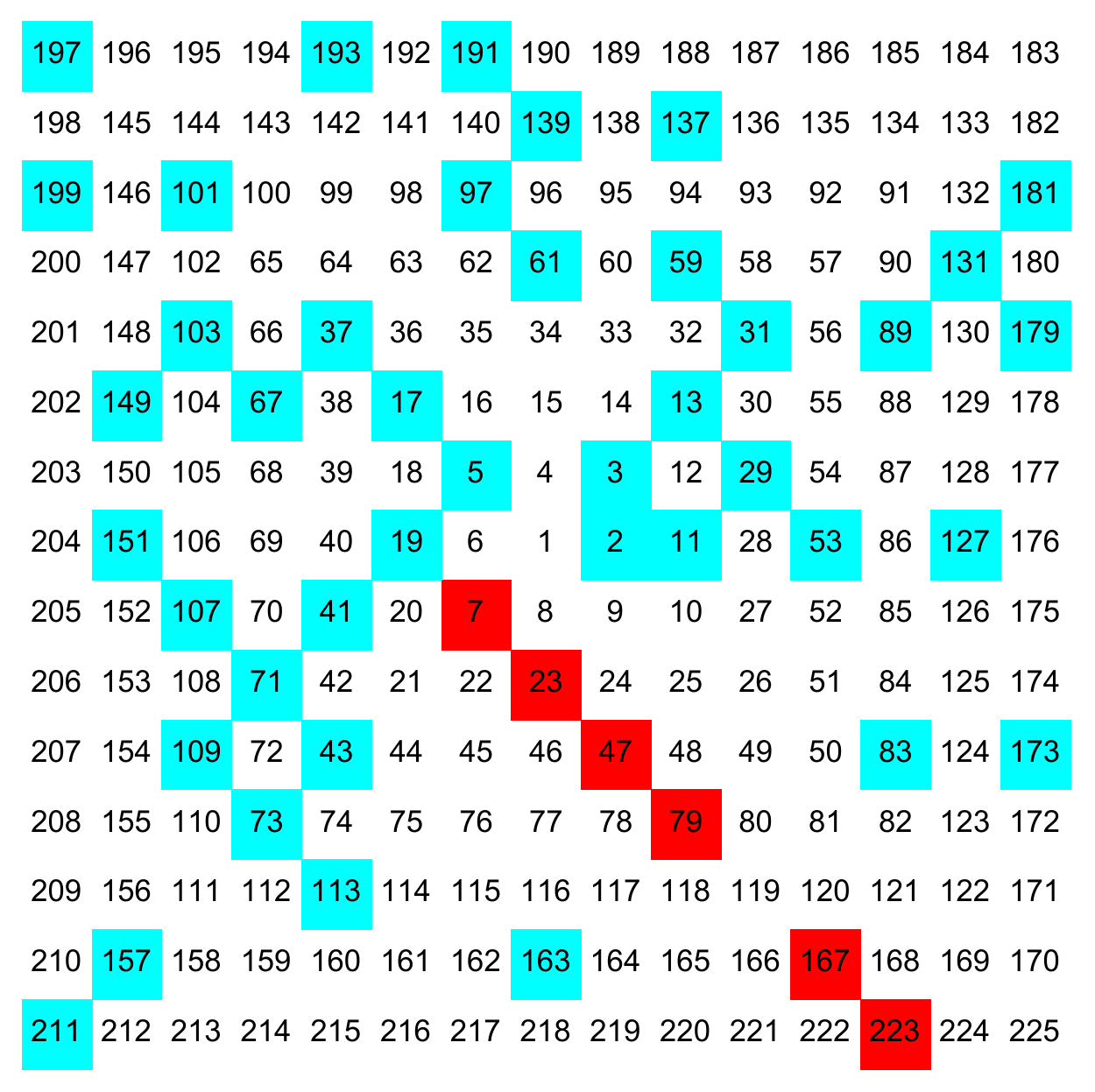}
    \caption{The diagonal ray $7,19,23,47,67,79\ldots$ contains an abundance of primes (red). 
    The $n$th number on the ray is $f(n) =4n^2+4n-1$.  The Bateman--Horn constant of this polynomial 
    is approximately $3.70$.}
    \end{subfigure}
    \quad
    \begin{subfigure}[t]{0.475\textwidth}
        \centering
        \includegraphics[width=\textwidth]{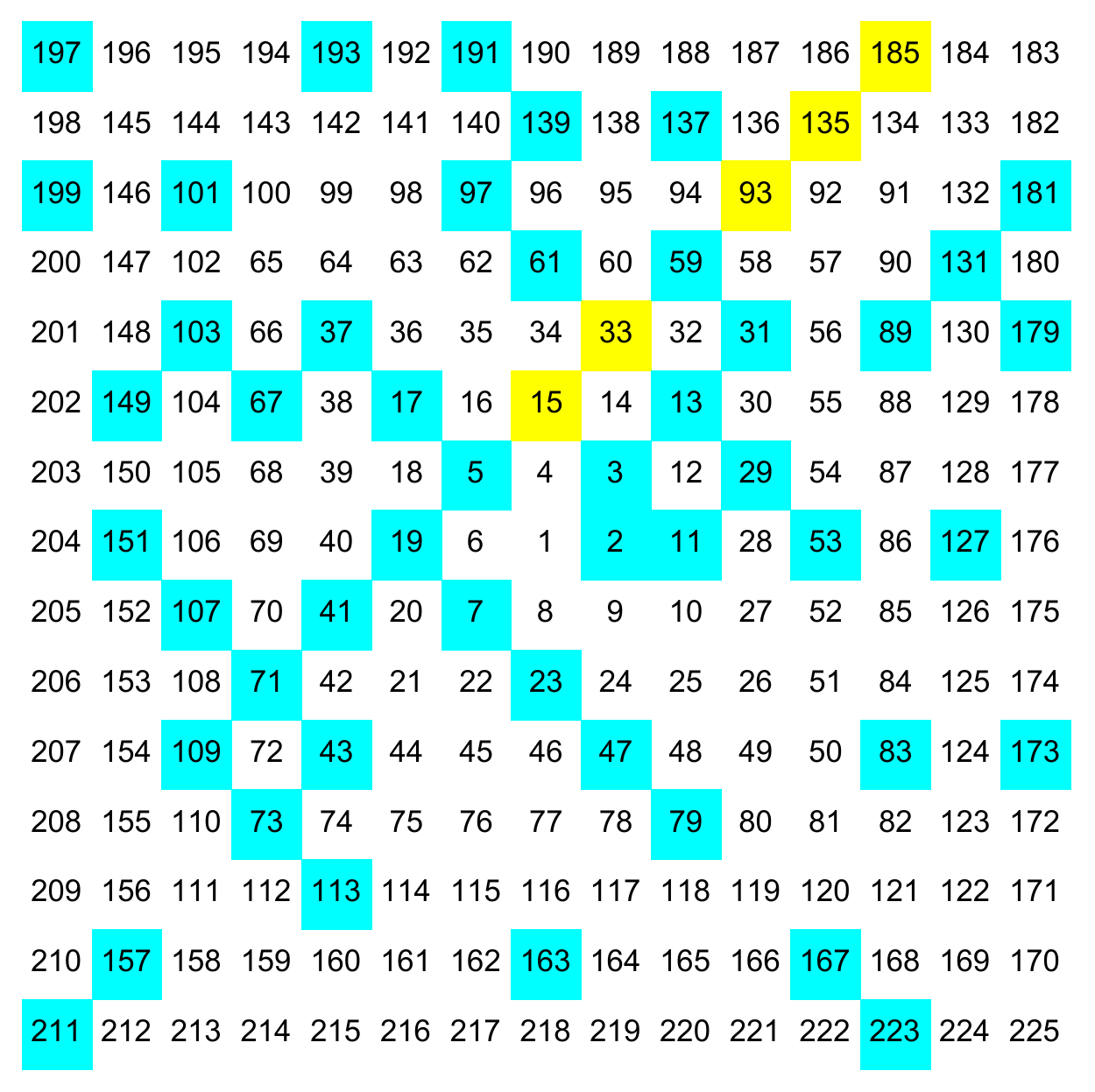}
    \caption{The diagonal ray $5,15,33,59,93,\ldots$ contains a few
    primes.  The $n$th number on the ray is
    $f(n) = 4n^2 - 2n+3$. The Bateman--Horn constant of this polynomial 
    is approximately $1.02$.}
    \end{subfigure}
    \caption{The relative number of primes on diagonal rays is governed by the Bateman--Horn conjecture.}
    \label{Figure:Diagonals}
\end{figure}

\begin{example}\label{Example:DiagonalBig}
Consider Figure \ref{Figure:Diagonals}a, in which the particularly prime-rich diagonal that includes 
the primes $7,19,23,47,67,79,103,167,199,223$ stands out.  As before, it is more convenient
to consider a single ray, in which the first differences increase monotonically.  We therefore study the ray
\begin{equation}\label{eq:NewList}
7,23,47,79,119,167,223,\ldots.
\end{equation}
Of these numbers only $119$ is composite.
If $f(n)$ denotes the $n$th number on the list \eqref{eq:NewList}, 
then an argument similar to that of the Example \ref{Example:HorizontalPrimeFree} shows that
\begin{equation*}
f(n) - f(n-1) = 8n
\end{equation*}
and hence
\begin{align}
f(n) &= \sum_{i=2}^n \big( f(i) - f(i-1) \big) + f(1) \nonumber \\
&=7+ \sum_{i=2}^n 8i \nonumber \\
&= 8 \left(\frac{n(n+1)}{2} - 1\right) + 7 \nonumber \\
&= 4n^2+4n-1.  \label{eq:4nn1}
\end{align}
Unlike \eqref{eq:UlamBad}, this polynomial is irreducible.  Since it has at most two roots modulo any prime and it does not
vanish identically modulo $2$, it does not vanish identically modulo any prime.  Consequently, the Bateman--Horn conjecture 
suggests that it assumes infinitely many prime values.  
Since the discriminant of the polynomial \eqref{eq:4nn1} is $32$, 
the general computation \eqref{eq:HLC} tells us that
\begin{equation*}
Q(f;x) \sim \frac{1}{2}C(f) \Li(x),
\end{equation*}
in which
\begin{equation*}
C(f) = 2 \prod_{p \geq 3} \left( 1 - \frac{(32/p)}{p-1} \right).
\end{equation*}
Among the odd primes at most $67$ we have
\begin{equation*}
\left( \frac{32}{p} \right)
=
\begin{cases}
1 & \text{if $p = 7,17,23,31,41,47$},\\
-1 & \text{if $p = 3,5,11,13,19,29,37,43,53,59,61,67$}.
\end{cases}
\end{equation*}
This substantial imbalance among the first few odd primes makes $C(f)$ unusually large and
explains the particularly prime-rich diagonal that corresponds to this polynomial.  In particular,
numerical computations suggest that $\frac{1}{2}C(f) \approx 3.70$.
\end{example}

\begin{example}
Consider the diagonal ray $5,15,33,59,93,135,185,\ldots$; see Figure \ref{Figure:Diagonals}b.
Although it contains some primes, it does not appear as prime rich as the ray from 
Example \ref{Example:DiagonalBig}.
Its values correspond to $f(t) = 4t^2 - 2t+3$, which has discriminant $-44$.  
Since $(-44/3) = (-44/5) = 1$, the primes $3$ and $5$ conspire to make $C(f)$
smaller; see \eqref{eq:HLC}.
The coefficient of $\Li(x)$ provided by the Bateman--Horn conjecture is 
approximately $1.02$.  This is substantially lower than in the previous example.
\end{example}

In summary, the patterns that Ulam observed can be explained as follows.
If we agree to omit the first several consecutive terms on a given ray, then 
there are integers $b$ and $c$ such that the $n$th number on the ray is
\begin{equation*}
f(n) = 4n^2 + bn + c.
\end{equation*}
If $b$ is even, then the ray is diagonal.
If $b$ is odd, then the ray is horizontal or vertical.
Certain combinations of $b$ and $c$ yield reducible polynomials; in these cases the ray contains
at most one prime.  Other combinations of $b$ and $c$ yield irreducible polynomials; 
the Bateman--Horn conjecture predicts the relative number of primes along each such ray.

\section{Multiple polynomials}\label{Section:Multiple}

We are now ready to apply the Bateman--Horn conjecture to families
of irreducible polynomials $f_1,f_2,\ldots,f_k \in \Z[x]$ with positive leading coefficients, no two of which are multiples of each other.  
Recall that the product $f = f_1 f_2 \cdots f_k$ 
should not vanish modulo any prime.  Then the conjecture predicts that the number $Q(f_1,f_2,\ldots,f_k;x)$
of $n \leq x$ for which $f_1(n), f_2(n),\ldots, f_k(n)$ are simultaneously prime
is asymptotic to
\begin{equation*}
\frac{C(f_1,f_2,\ldots,f_k) }{\prod_{i=1}^k \deg f_i} 
\int _{2}^{x}\frac{dt}{(\log t)^k},
\end{equation*}
in which 
\begin{equation*}
C(f_1,f_2,\ldots,f_k) \ = \ \prod_p \left( 1 - \frac{1}{p} \right)^{-k}
\left(1- \frac{\omega_f(p)}{p} \right).
\end{equation*}
In particular, observe that the number $k$ of polynomials involved appears in the exponents
that occur in the integrand and the product that defines the Bateman--Horn constant.

\subsection{Twin prime conjecture}\label{Section:Twin}
If $p$ and $p+2$ are prime, then $p$ and $p+2$ are \emph{twin primes}.  
The long-standing twin prime conjecture asserts that there are infinitely many twin primes.
Although this question likely puzzled thinkers since Euclid's time, the earliest extant record
of the conjecture (in a more general form, see Section \ref{Section:Cousin}) is from Alphonse de Polignac (1826--63) in 1849.
While it remains unproven, recent years have seen an explosion of closely-related work
\cite{Polymath, Zhang, Maynard}.

In 1919, Viggo Brun (1885-- 1978) proved that the sum
\begin{equation}\label{eq:BrunSum}
\left( \frac{1}{3} + \frac{1}{5}\right) + \left( \frac{1}{5} + \frac{1}{7} \right)
+ \left( \frac{1}{11} + \frac{1}{13} \right) + \left( \frac{1}{17} + \frac{1}{19} \right) + \cdots
\end{equation}
of the reciprocals of the twin primes converges.  This stands in stark contrast to 
Euler's discovery that $\sum_p 1/p$ diverges (Theorem \ref{Theorem:Euler}).  Thus,
the twin primes must be far sparser, in the sense of reciprocal sums, than the primes themselves.  
The sum \eqref{eq:BrunSum}, which is now known as \emph{Brun's constant}, 
is greater than $1.83$ and less than $2.347$ \cite{KlyveThesis}
(numerical evidence suggests a value
of approximately $1.9$).

\begin{figure}
    \centering
    \begin{subfigure}[t]{0.475\textwidth}
        \centering
        \includegraphics[width=\textwidth]{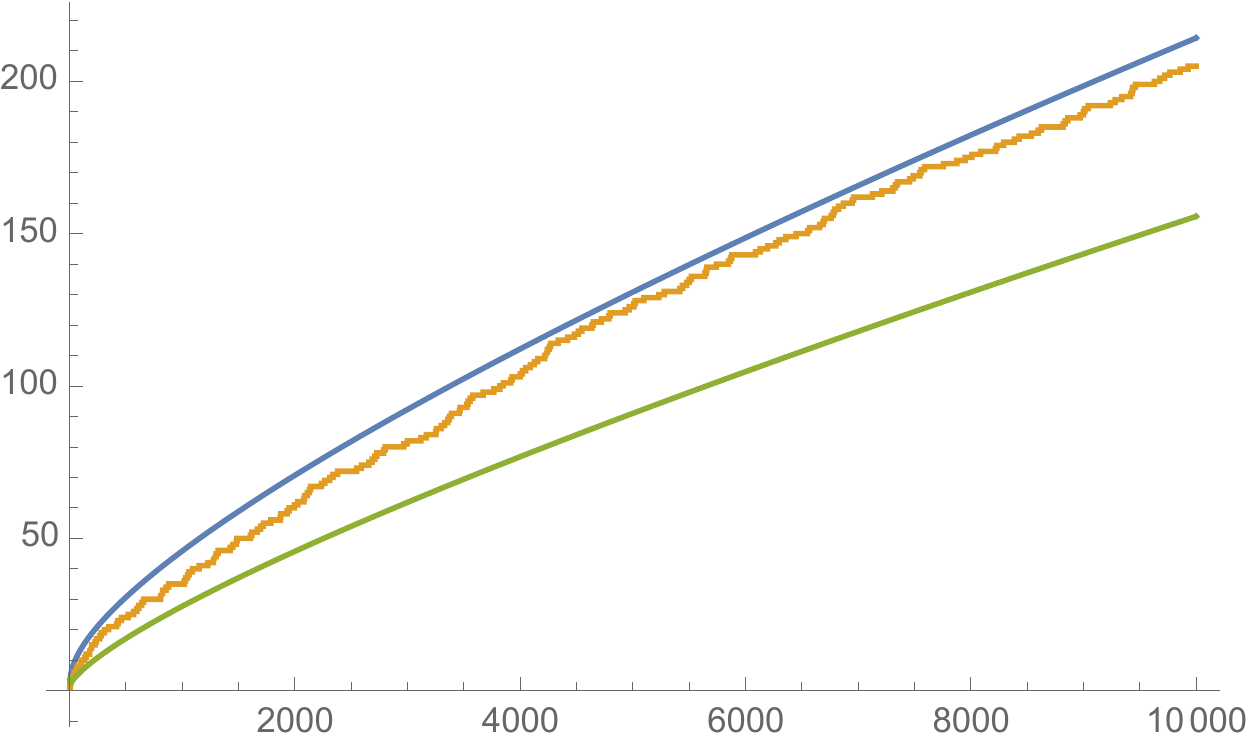}
    \caption{$x \leq 10{,}000$}
    \end{subfigure}
    \quad
    \begin{subfigure}[t]{0.475\textwidth}
        \centering
        \includegraphics[width=\textwidth]{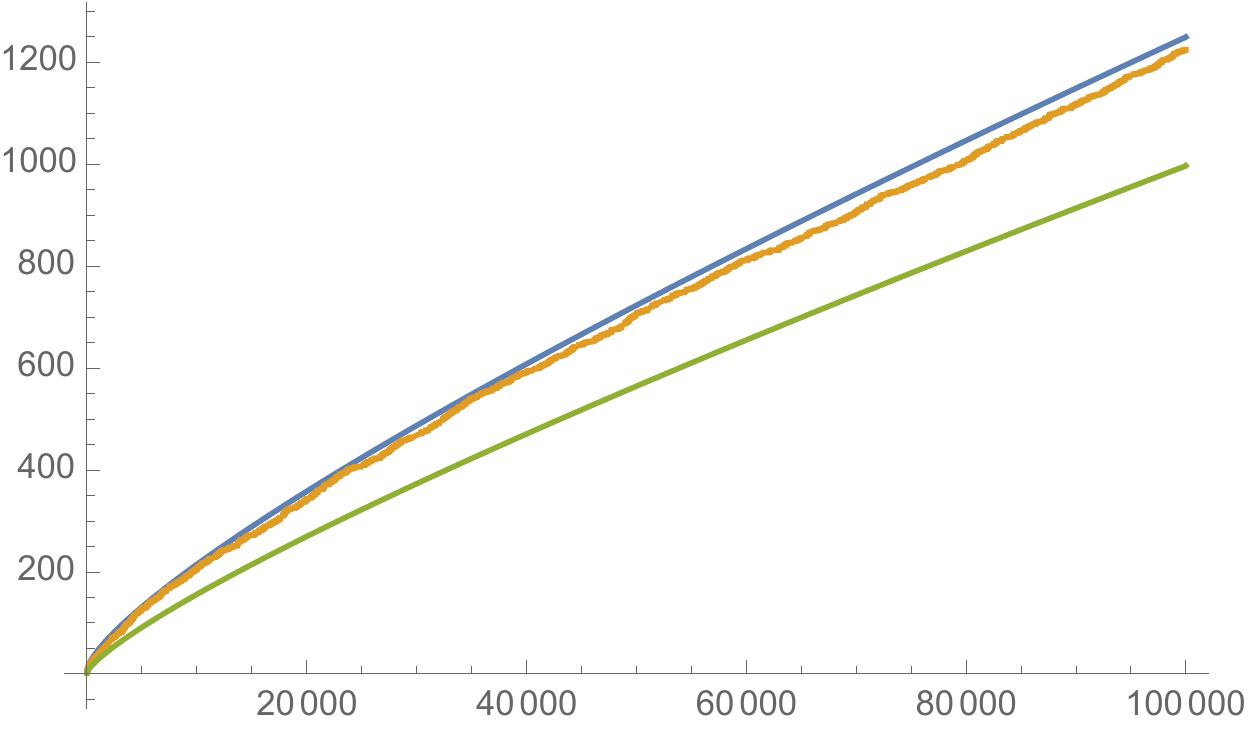}
    \caption{$x \leq 100{,}000$.}
    \end{subfigure}
    \caption{Graph of $\pi_2(x)$ (orange) versus $2C_2 \int_2^x (\log t)^{-2} \,dt$ (blue) and $2C_2 x / (\log x)^2$ (green).
    The more complicated integral expression apparently provides a much better approximation than does the simpler expression.
    }
    \label{Figure:TwinPrimes}
\end{figure}

What does the Bateman--Horn conjecture have to say about twin primes?
Let $f_1(t) = t$ and $f_2(t) = t+2$.  Then $f_1(t)$ and $f_2(t)$ are simultaneously prime
if and only if $t$ is the lesser element of a twin-prime pair.
Let $f = f_1 f_2$.
Then
\begin{equation*}
f(t) \equiv 0 \pmod{p}
\qquad \iff \qquad
t(t-2) \equiv 0 \pmod{p},
\end{equation*}
and hence
\begin{equation*}
\omega_f(p) = 
\begin{cases}
1 & \text{if $p=2$},\\
2 & \text{if $p \geq 3$}.
\end{cases}
\end{equation*}
The corresponding Bateman--Horn constant is
\begin{align*}
C(f_1,f_2)
&= \prod_p \left(1 - \frac{1}{p} \right)^{-2} \left(1 - \frac{\omega_f(p)}{p}\right) \\
&= 2\prod_{p\geq 3} \frac{p^2}{(p-1)^2} \cdot \frac{p-2}{p}\\
&= 2 \prod_{p \geq 3} \frac{p(p-2)}{(p-1)^2} \\
&= 2 C_2,
\end{align*}
in which
\begin{equation*}
C_2 = \prod_{p \geq 3} \frac{p(p-2)}{(p-1)^2} \approx  0.660161815
\end{equation*}
is the \emph{twin primes constant}.  
The Bateman--Horn conjecture predicts that
\begin{equation*}
Q(f_1,f_2;x)
\sim 2C_2 \int _{2}^{x}{dt \over (\log  t)^{2}}.
\end{equation*}
It is more traditional to express this in terms of the twin prime counting function.
Let $\pi_2(x)$ denote the number of primes $p$ at most $x$ for which $p+2$ is prime.
Then $\pi_2(x) = Q(f_1,f_2;x)$ and (by Lemma \ref{Lemma:Log})
\begin{equation*}
 \pi _{2}(x)\,\sim\, 2C_{2}\int _{2}^{x}{dt \over (\log  t)^{2}} 
 \,\sim\, \frac{2C_2x}{(\log x)^2};
\end{equation*}
see Figure \ref{Figure:TwinPrimes}.
This asymptotic estimate for $\pi_2(x)$ was first postulated by Hardy and Littlewood \cite{Hardy}.

\subsection{Cousin primes, sexy primes, and more}\label{Section:Cousin}

If $p$ and $p+4$ are prime, then $p$ and $p+4$ are \emph{cousin primes}.
If $p$ and $p+6$ are prime, then $p$ and $p+6$ are \emph{sexy primes}.
Thankfully the nomenclature appears to expire after this point, although it is still fruitful to consider
prime pairs $p,p+k$, in which $k\geq 2$ is even.  

Alphonse de Polignac conjectured in 1849 that for each even number $k$,
there are infinitely many prime pairs $p,p+k$.  This is now known as \emph{Polignac's conjecture}.
The case $k=2$ of Polignac's conjecture is the twin prime conjecture (Section \ref{Section:Twin}),
which remains unproven.
In light of the work of Yitang Zhang (1955--) \cite{Zhang} and the Polymath8b Project \cite{Polymath} on bounded gaps
between primes, we know that there is an even $k \leq 246$ for which 
infinitely many prime pairs $p,p+k$ exist.  Unfortunately, we do not 
know a specific value of $k$ for which this occurs.

The Bateman--Horn conjecture goes much further than even Polignac's conjecture.  It implies the existence of infinitely many pairs $p,p+k$ of primes
for each even $k$ and also supplies asymptotic predictions that are backed up by numerical computations.  
The following calculations were worked out in \cite{GLS}.
Let $f_1(t) = t$ and $f_2(t) = t+k$, and let $f = f_1f_2$.  Then
\begin{equation*}
f(t) \equiv 0 \pmod{p}
\qquad \iff \qquad
t(t+k) \equiv 0 \pmod{p},
\end{equation*}
and hence
\begin{equation*}
\omega_f(p) = 
\begin{cases}
1 & \text{if $p|k$},\\
2 & \text{if $p \nmid k$}.
\end{cases}
\end{equation*}
The Bateman--Horn constant is
\begin{align}
C(f_1,f_2;x)
&= \prod_p \left(1 - \frac{1}{p}\right)^{-2} \left(1 - \frac{\omega_f(p)}{p}\right) \nonumber \\
&= \prod_{p|k} \left(1 - \frac{1}{p}\right)^{-1} \prod_{p \nmid k} \left(1 - \frac{1}{p}\right)^{-2} \left(1 - \frac{2}{p}\right) \nonumber \\
&= \prod_{p|k} \frac{p}{p-1} \prod_{p \nmid k} \frac{p(p-2)}{(p-1)^2} . \label{eq:PPBH}
\end{align}
To highlight the dependence on $k$ and match the historically established notation in
the twin prime setting (Section \ref{Section:Twin}),
we denote the preceding constant by $2C_k$; that is,
\begin{equation}\label{eq:Ck}
C_k = \prod_{\substack{p|k \\ p \geq 3}}  \frac{p}{p-1} \prod_{p \nmid k} \frac{p(p-2)}{(p-1)^2}.
\end{equation}
We do not define $C_k$ for odd $k$; this would be pointless since for each odd 
$k$ there is at most one prime pair $p,p+k$.
Since $\sum_p 1/p^2$ converges,
the infinite product \eqref{eq:Ck}  that defines $C_k$ converges absolutely since
\begin{equation}\label{eq:pp21p12}
\frac{p(p-2)}{(p-1)^2} = 1 - \frac{1}{(p-1)^2}.
\end{equation}
Numerical approximations for $C_2,C_4,\ldots,C_{150}$ are given in Table \ref{Table:CkList}.
If $\pi_k(x)$ denotes the number of primes $p \leq x$ for which $p+k$ is prime,
then the Bateman--Horn conjecture predicts that
\begin{equation*}
\pi_k(x) \, \sim \, 2C_k \int_2^x \frac{dt}{(\log t)^2} \, \sim \, \frac{2C_k x}{(\log x)^2}.
\end{equation*}

\begin{table}\small
\begin{equation*}
\begin{array}{cl|cl|cl|cl|cl}
k & \multicolumn{1}{c|}{C_k} & k & \multicolumn{1}{c|}{C_k} & k & \multicolumn{1}{c|}{C_k} & k & \multicolumn{1}{c|}{C_k} & k & \multicolumn{1}{c}{C_k} \\
\hline
 2 & 0.660162 & 32 & 0.660162 & 62 & 0.682926 & 92 & 0.691598 & 122 & 0.671351 \\
 4 & 0.660162 & 34 & 0.704173 & 64 & 0.660162 & 94 & 0.674832 & 124 & 0.682926 \\
 6 & 1.32032 & 36 & 1.32032 & 66 & 1.46703 & 96 & 1.32032 & 126 & 1.58439 \\
 8 & 0.660162 & 38 & 0.698995 & 68 & 0.704173 & 98 & 0.792194 & 128 & 0.660162 \\
 10 & 0.880216 & 40 & 0.880216 & 70 & 1.05626 & 100 & 0.880216 & 130 & 0.960235 \\
 12 & 1.32032 & 42 & 1.58439 & 72 & 1.32032 & 102 & 1.40835 & 132 & 1.46703 \\
 14 & 0.792194 & 44 & 0.733513 & 74 & 0.679024 & 104 & 0.720177 & 134 & 0.670318 \\
 16 & 0.660162 & 46 & 0.691598 & 76 & 0.698995 & 106 & 0.673106 & 136 & 0.704173 \\
 18 & 1.32032 & 48 & 1.32032 & 78 & 1.44035 & 108 & 1.32032 & 138 & 1.3832 \\
 20 & 0.880216 & 50 & 0.880216 & 80 & 0.880216 & 110 & 0.978018 & 140 & 1.05626 \\
 22 & 0.733513 & 52 & 0.720177 & 82 & 0.677089 & 112 & 0.792194 & 142 & 0.669729 \\
 24 & 1.32032 & 54 & 1.32032 & 84 & 1.58439 & 114 & 1.39799 & 144 & 1.32032 \\
 26 & 0.720177 & 56 & 0.792194 & 86 & 0.676263 & 116 & 0.684612 & 146 & 0.66946 \\
 28 & 0.792194 & 58 & 0.684612 & 88 & 0.733513 & 118 & 0.671744 & 148 & 0.679024 \\
 30 & 1.76043 & 60 & 1.76043 & 90 & 1.76043 & 120 & 1.76043 & 150 & 1.76043 \\
\end{array}
\end{equation*}
\caption{Numerical approximations of the constants $C_k$ based upon the first $1{,}000{,}000$
terms of the product \eqref{eq:PPBH}.}
\label{Table:CkList}
\end{table}

There are several important observations to make.  
\begin{itemize}
\item The conjectured rate of growth in $\pi_k$ depends only upon the constant $C_k$.
Furthermore, $C_k$ depends only upon the primes that divide $k$.

\item In light of \eqref{eq:pp21p12},
an examination of \eqref{eq:PPBH} reveals that $C_k$ is minimized when $k$ is a power of two,
in which case $C_2 = C_4 = C_8 = C_{16} = \cdots  \approx 0.660162$.

\item $\lim_{p\to\infty} C_{2p} = C_2$.  That is, $C_k$ can be made arbitrarily close 
to the twin primes constant $C_2$ by letting $k=2p$, in which $p$ is a sufficiently
large prime.

\item $C_k$ can be made arbitrarily large by selecting $k$ to have
sufficiently many small prime factors.  The first factor in \eqref{eq:PPBH} is
\begin{equation*}
\prod_{\substack{p|k\\ p \geq 3}} \frac{p}{p-1}
=\prod_{\substack{p|k\\ p \geq 3}} \left(1 + \frac{1}{p-1} \right).
\end{equation*}
If $k$ is the product of the first $n$ primes (that is, $k$ is the $n$th \emph{primorial} $p_n\#$),
then the preceding diverges as $n \to \infty$.
\end{itemize}


The patterns predicted by the Bateman--Horn conjecture
are evident in Table \ref{Figure:Pis}, which provides the numerical values of $\pi_k(10^n)$ for several $k$
and $n=2,3,\ldots,8$.  For example, Table \ref{Table:CkList} suggests that
primes $p$ for which $p+30$ is prime should be about
\begin{equation*}
\frac{1.76043}{0.660162}
\approx 2.6667
\end{equation*}
times more numerous than twin primes.  Among the first $10^8$ primes, Table \ref{Figure:Pis}
gives the proportion
\begin{equation*}
\frac{17{,}331{,}689}{6{,}497{,}407}
\approx 2.66748.
\end{equation*}
The agreement is remarkable.

\begin{table}\small
\begin{equation*}
\begin{array}{c|>{\columncolor{blue!10}}c>{\columncolor{blue!10}}c>{\columncolor{green!10}}c>{\columncolor{blue!10}}c>{\columncolor{yellow!10}}c>{\columncolor{green!10}}c>{\columncolor{orange!10}}c}
n & \pi_2(p_{10^n}) & \pi_4(p_{10^n}) & \pi_6(p_{10^n}) & \pi_8(p_{10^n}) & \pi_{10}(p_{10^n}) & \pi_{12}(p_{10^n}) & \pi_{30}(p_{10^n}) \\
\hline
2 & 25 & 27 & 48 & 24 & 33 & 48 & 61 \\
3 & 174 & 170 & 343 & 178 & 230 & 340 & 456 \\
4 & 1{,}270 & 1{,}264 & 2{,}538 & 1{,}303 & 1{,}682 & 2{,}515 & 3{,}450 \\
5 & 10{,}250 & 10{,}214 & 20{,}472 & 10{,}336 & 13{,}653 & 20{,}462 & 27{,}434 \\
6 & 86{,}027 & 85{,}834 & 170{,}910 & 85{,}866 & 114{,}394 & 171{,}618  & 228{,}548 \\
7  & 738{,}597 & 738{,}718 & 1{,}477{,}321 & 738{,}005 & 984{,}809 & 1{,}477{,}496 & 1{,}970{,}049 \\
8  & 6{,}497{,}407 & 6{,}496{,}372 & 12{,}992{,}625 & 6{,}497{,}273 & 8{,}667{,}364 &
12{,}994{,}918 & 17{,}331{,}689\\
\end{array}
\end{equation*}
\normalsize
\caption{Values of the counting functions $\pi_k(x)$ at $p_{10^n}$, the $10^n$th prime.
The asymptotic predictions of Bateman--Horn conjecture are identical
for $\pi_2$, $\pi_4$, and $\pi_8$ (blue), and for $\pi_6$ and $\pi_{12}$ (green).
The computations appear to corroborate this.  
}
\label{Figure:Pis}
\end{table}

\subsection{Sophie Germain primes}

A prime number $p$ is a \emph{Sophie Germain prime} if $2p+1$ is also a prime. Such primes were first introduced and investigated by the legendary French mathematician, physicist and philosopher Marie-Sophie Germain (1776--1831) in the course of her work on some early cases of Fermat's Last Theorem; see 
\cite[Sect.5.5.5]{shoup} for further information. 

If $p$ is a Sophie Germain prime, then $2p+1$ is the corresponding \emph{safe prime}. 
This terminology reflects the usefulness of such primes in cryptography. Specifically, the famous RSA (Rivest--Shamir--Adleman) cryptosystem is an asymmetric cryptoscheme using a public key to encrypt a message and a private key to decrypt it \cite{crypto}. 
The public key is a product of two large prime numbers
(for example, a product of two safe primes)
and the hardness of a hostile attack is based on the difficulty of factoring such a product.  Factorization is especially difficult if 
the primes in question are of comparable size.
Cryptographic applications provide a strong modern motivation for studying such prime numbers, and it is conjectured that there are infinitely many Sophie Germain (and hence safe) primes. This conjecture is currently open, and the largest Sophie Germain prime known has $51780$ digits \cite{jarai}.

The search for Sophie Germain primes can be rephrased in the language of the Bateman--Horn conjecture. Let $f_1(t) = t$ and $f_2(t) = 2t+1$.  
Then $p$ is a Sophie Germain prime if and only if $f_1(p)$ and $f_2(p)$ are simultaneously prime. 
The infinitude of these primes follows from the Bateman--Horn conjecture, which
also provides an asymptotic estimate on their counting function. 
The polynomial
\begin{equation*}
f(t) = f_1(t) f_2(t) = t(2t+1)
\end{equation*}
does not vanish identically modulo any prime since $f(1) \equiv 1 \pmod{2}$
and $f$ has at most two roots modulo any odd prime.  Since $f$ vanishes at $0$ and 
$(p-1)/2$ for every odd prime $p$, we deduce that
\begin{equation*}
\omega_f(p) \ = \ 
\begin{cases}
1 & \text{if $p = 2$},\\
2 & \text{if $p$ is odd}.
\end{cases}
\end{equation*}
Thus,
\begin{equation*}
C(f_1,f_2) = 2 \prod_{p \neq 2} \left(1 - \frac{1}{p} \right)^{-2} \left( 1 - \frac{2}{p} \right) = 2 \prod_{p \neq 2} \frac{p (p-2)}{(p-1)^2}  \approx 1.32032\ldots.
\end{equation*}
Since $\deg f_1 = \deg f_2 = 1$, we obtain the estimate
\begin{equation*}
Q(f_1,f_1;x) \ \sim \ (1.32032\ldots) \int _{2}^{x}\frac{dt}{(\log t)^2}.
\end{equation*}
This is the same asymptotic prediction as in the twin-prime case (Section \ref{Section:Twin}).

\subsection{Cunningham chains}

A sequence $p_1,p_2,\ldots,p_n$ of primes is a \emph{Cunningham chain} 
of the first kind if $p_{i+1} = 2p_i+1$ for each $1 \leq i \leq n-1$ and of the second kind if $p_{i+1} = 2p_i-1$.
That is, every $p_i$ in a Cunningham chain of the first kind, except for $p_n$, is a Sophie Germain prime 
and every $p_i$, except for $p_1$, is a safe prime. 
Cunningham chains are named after a British mathematician Allan Joseph Champneys Cunningham (1842 -- 1928) who first introduced and studied them \cite{cunningham}.
Here are a few examples of Cunningham chains of the first kind
\begin{equation*}
(2, 5, 11, 23, 47),\ (3, 7),\ (89, 179, 359, 719, 1439, 2879),
\end{equation*}
and of the second kind
\begin{equation*}
(2, 3, 5),\ (7, 13),\ (19, 37, 73).
\end{equation*}
The longest known Cunningham chains have length $19$ \cite{augustin}. 

The existence of arbitrary long Cunningham chains follows from the first Hardy--Littlewood conjecture, and hence from Bateman--Horn conjecture. Indeed, let
\begin{equation*}
f_1(t) = t,\quad f_2(t) = 2f_1(t) \pm 1, \ldots,\quad f_k(t) = 2f_{k-1}(t) \pm 1,
\end{equation*}
then we need them all to be prime simultaneously. Bateman--Horn guarantees the existence of infinitely many such $k$-tuples and even gives an asymptotic estimate on the growth of their number 
that is analogous to the argument above for Sophie Germain primes. 
On the other hand, it has been proved that a Cunningham chain of infinite length cannot exist. 
Indeed, suppose for instance that odd primes $p_1,p_2,\dots$ form a Cunningham chain of the first kind.  Then
$$p_{i+1} = 2p_i + 1 = 2(2p_{i-1}+1)+1 = \dots = 2^i p_1 + \sum_{j=0}^{i-1} 2^j = 2^i p_1 + (2^i-1)$$
and hence $p_{i+1} \equiv 2^i - 1\ (\md p_1)$. On the other hand, Fermat's little theorem implies that
$$2^{p_1-1} -1 \equiv 0\ (\md p_1),$$
meaning that $p_{p_1}$ would be divisible by~$p_1$, and so cannot be prime. This implies that, in fact, a Cunningham chain 
starting with an odd prime $p_1$ cannot have more than~$p_1-1$ terms in it. If $p_1 = 2$, then the same argument can be applied
to the chain $p_2,p_3,\dots$. Further information about Cunningham chains and their use in cryptography can be found in \cite{buhler}.

\subsection{Green--Tao theorem}

One of the most spectacular results in twenty-first century number theory is the Green--Tao theorem \cite{GreenTao},
which asserts that the primes contain arbitrarily long arithmetic progressions.  That is, given $k \geq 1$ there is a $k$-term
arithmetic progression 
\begin{equation*}
b, \quad b+ a,\quad  b+2a,\ldots,\quad b+(k-1)a
\end{equation*}
of prime numbers.  For example, $5,11,17,23,29$ is a $5$-term arithmetic
progression of primes with $b = 5$ and $a = 6$.

Consider the $k$ linear polynomials
\begin{equation*}
f_1(t) = t, \qquad f_2(t) = t+a, \ldots,
\quad f_k(t) = t + (k-1)a,
\end{equation*}
each of which is obviously irreducible.
Let $f= f_1f_2\cdots f_k$ denote their product.
The congruence $f(t) \equiv 0 \pmod{p}$ is
\begin{equation*}
t(t+a)(t+2a)\cdots\big(t+(k-1)a\big) \equiv 0 \pmod{p}.
\end{equation*}
Thus,
\begin{equation*}
\omega_f(p) = 
\begin{cases}
1 & \text{if $p|a$},\\
\min\{k,p\} & \text{if $p \nmid a$}.
\end{cases}
\end{equation*}
If $p \leq k$ and $p\nmid a$, then $f$ vanishes identically modulo $p$.
Consequently, we require that $p|a$ for all primes $p \leq k$.
This suggests that we take $a = p_k\#$, the product of the first $k$ prime numbers.
Then
\begin{equation*}
Q(f_1,f_2,\ldots,f_k;x)
\ \sim \ C(f_1,f_2,\ldots,f_k) \int_2^x \frac{dt}{(\log t)^k},
\end{equation*}
in which
\begin{equation*}
C(f_1,f_2,\ldots,f_k)
= \prod_{n=1}^k \left(1 - \frac{1}{p_n} \right)^{-k+1}
\prod_{n=k+1}^{\infty} \left(1 - \frac{1}{p} \right)^{-k} \left(1 - \frac{k}{p} \right) 
\end{equation*}
is a nonzero constant.  This yields the following famous result
\cite[Thm.~1.1]{GreenTao}.

\begin{theorem}[Green--Tao, 2004]
For each positive integer $k$,
the prime numbers contain infinitely many arithmetic progressions of
length $k$.
\end{theorem}

\section{Limitations of the Bateman--Horn conjecture}\label{Section:Limitations}
Although we have touted the Bateman--Horn conjecture as ``one conjecture to rule them all,'' it has its limitations.  We briefly discuss a number of topics in number theory that the conjecture
does not appear to address.

First of all, the Bateman--Horn conjecture is a statement about the overall distribution of prime numbers.  It says little about what happens on small scales.
For example, it does not appear to resolve 
Legendre's conjecture (for each $n$ there is a prime between $n^2$ and $(n+1)^2$).
Bateman--Horn also does not tell us much about the additive properties of the prime numbers.  For instance, it does not seem to imply the Goldbach conjecture (every even number greater than $4$ 
is the sum of two odd prime numbers).

The Bateman--Horn conjecture does an excellent job predicting the asymptotic
distribution of primes generated by families of polynomials.
However, it does not tell us much about primes generated by non-polynomial functions.
For example, it has nothing to say about the number of primes of the form
$2^{2^n}+1$ (Fermat primes) or $2^n-1$ (Mersenne primes).

The conjecture has little to say about diophantine equations, such as the Fermat equation $x^n+y^n=z^n$ \cite{StewartTall}
or the Catalan equation $x^n-y^m = 1$ \cite{Mihailescu}.
For example, the Bateman--Horn conjecture appears to have little overlap with the
$abc$-conjecture and its applications; see \cite[Ch.~11]{luca} or \cite[Ch.~12]{bombieri} for a detailed overview of the far-reaching $abc$-conjecture and its numerous connections.

The Bateman--Horn conjecture provides asymptotics for counting functions
related to primes, but does not bound the size of the error terms.  For example, it implies the prime number theorem
(Theorem \ref{Theorem:PNT}), which asserts that $\pi(x) \sim \li(x)$.
However, BH does not tell us about $|\pi(x) - \li(x)|$.  On the other hand, Schoenfeld \cite{Schoenfeld}
proved that the Riemann hypothesis yields
\begin{equation*}
|\pi(x) - \operatorname{li}(x)| < \frac{1}{8 \pi}\sqrt{x} \log x,\qquad x\geq 2{,}657.
\end{equation*}
Thus, the Riemann hypothesis implies the prime number theorem with a well-controlled error term.
Serge Lang says:
\begin{quote}\small
I regard it as a major problem to give an estimate for the error term in the Bateman--Horn conjecture similar to the Riemann hypothesis.
This could possibly lead to a vast reconsideration of the context for Riemann's explicit formulas. 
\cite[p.11]{Lang}.
\end{quote}

Number theory is one of the central branches of mathematics and connects with analysis, algebra, combinatorics, and many other fields.
It has enjoyed a great number of exciting advances and breakthroughs in recent years, several of which have led to Fields medals and other prestigious awards. 
It also contains a great number of difficult and deep open problems and conjectures.
To a large extent these influence the course of modern mathematics. 
Some problems, like the Riemann hypothesis, the $abc$-conjecture, the twin prime conjecture, or the Goldbach conjecture 
are well known and rightfully celebrated by the mathematical community. 
Others, like the Bateman--Horn conjecture, although of equally great stature, are not as well known. 
The goal of this paper was to present an overview of this important problem, its connections, and its consequences. 
It is our hope that we have convinced the reader that the Bateman--Horn conjecture deserves to be ranked among the most pivotal unproven conjectures in the theory of numbers.

\bibliographystyle{amsplain} 
\bibliography{BHCHHA}

\end{document}